\documentclass[12pt,twoside]{amsart}
\usepackage[margin=3cm]{geometry}
\usepackage[colorlinks=false]{hyperref}
\usepackage[english]{babel}
\usepackage{graphicx,titling}
\usepackage{float}
\usepackage{amsmath,amsfonts,amssymb,amsthm}
\usepackage{lipsum}
\usepackage[T1]{fontenc}
\usepackage{fourier}
\usepackage{color}
\usepackage{esint}
\usepackage{caption}
\usepackage{ccicons}

\makeatletter
\def\blfootnote{\xdef\@thefnmark{}\@footnotetext}
\makeatother

\newcommand\ccnote{
    \blfootnote{\copyright\,\, Mihaela Ifrim, James Rowan, Daniel Tataru and Lizhe Wan}
    \blfootnote{\ccLogo\, \ccAttribution\,\, Licensed under a \href{https://creativecommons.org/licenses/by/4.0/}{Creative Commons Attribution License (CC-BY)}.}
}

\usepackage[export]{adjustbox}
\numberwithin{equation}{section}
\usepackage{setspace}
\renewcommand{\le}{\leqslant}
\renewcommand{\leq}{\leqslant}
\renewcommand{\ge}{\geqslant}
\renewcommand{\geq}{\geqslant}
\renewcommand{\mathbb}{\varmathbb}
\usepackage{fancyhdr}
\newtheorem{theorem}{Theorem}[section]

\newtheorem{proposition}[theorem]{Proposition}
\newtheorem{definition}[theorem]{Definition}
\newtheorem{remark}[theorem]{Remark}
\usepackage{amsmath,amscd,amsbsy,amssymb,amsthm,amsfonts}
\usepackage{latexsym,url,bm}
\usepackage{cite,enumitem}
\usepackage{xcolor}
\usepackage{tikz}

\newcommand{\doth}[1]{\Dot{\mathcal{H}}^{#1}}

\newcommand{\W}{{\mathbf W}}
\newcommand{\nP}{{\mathbf P}}
\newcommand{\uA}{{\underline A}}
\newcommand{\uB}{{\underline B}}

\usepackage{fullpage}

\newtheorem*{thm}{Theorem}

\address{Mihaela Ifrim, Department of Mathematics, University of Wisconsin - Madison}
\email{ifrim@wisc.edu}
\medskip

\address{James Rowan, Department of Mathematics, University of California at Berkeley}
\email{jrowan@math.berkeley.edu}
\medskip

\address{Daniel Tataru, Department of Mathematics, University of California at Berkeley}
\email{tataru@math.berkeley.edu}
\medskip

\address{Lizhe Wan, Department of Mathematics, University of Wisconsin - Madison}
\email{lwan33@wisc.edu}

\begin{document}

\thispagestyle{empty}

\begin{minipage}{0.28\textwidth}
\begin{figure}[H]
\includegraphics[width=2.5cm,height=2.5cm,left]{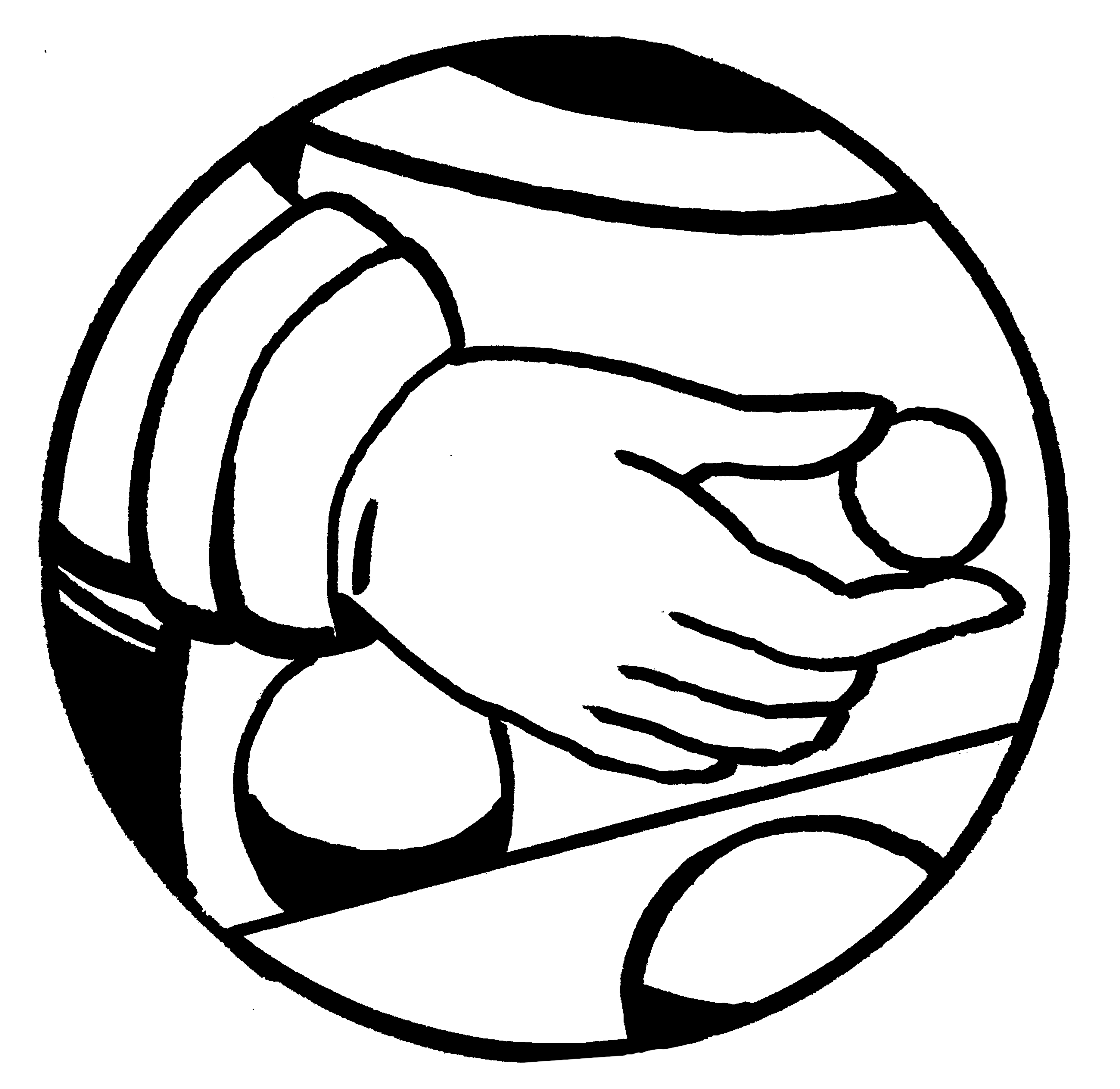}
\end{figure}
\end{minipage}
\begin{minipage}{0.7\textwidth} 
\begin{flushright}
Ars Inveniendi Analytica (2022), Paper No. 3, 32 pp.
\\
DOI 10.15781/chkn-sn69
\\
ISSN: 2769-8505
\end{flushright}
\end{minipage}

\ccnote

\vspace{1cm}


\begin{center}
\begin{huge}
\textit{The Benjamin-Ono Approximation for 2D Gravity Water Waves with Constant Vorticity}

\end{huge}
\end{center}

\vspace{1cm}


\begin{minipage}[t]{.28\textwidth}
\begin{center}
{\large{\bf{Mihaela Ifrim}}} \\
\vskip0.15cm
\footnotesize{University of Wisconsin, Madison}
\end{center}
\end{minipage}
\hfill
\noindent
\begin{minipage}[t]{.28\textwidth}
\begin{center}
{\large{\bf{James Rowan}}} \\
\vskip0.15cm
\footnotesize{University of California, Berkeley}
\end{center}
\end{minipage}
\hfill
\noindent
\begin{minipage}[t]{.28\textwidth}
\begin{center}
{\large{\bf{Daniel Tataru}}} \\
\vskip0.15cm
\footnotesize{University of California, Berkeley} 
\end{center}
\end{minipage}
\hfill
\noindent

\begin{center}
\begin{minipage}[t]{.28\textwidth}
\begin{center}
{\large{\bf{Lizhe Wan}}} \\
\vskip0.15cm
\footnotesize{University of Wisconsin, Madison} 
\end{center}
\end{minipage}
\hfill
\end{center}
\vspace{1cm}


\begin{center}
\noindent \em{Communicated by Frank Merle}
\end{center}
\vspace{1cm}


\noindent \textbf{Abstract.} \textit{
This article is concerned with infinite depth gravity water waves with constant vorticity in two space dimensions. 
We consider this system expressed in position-velocity potential holomorphic coordinates. 
 We show that, for low-frequency solutions, the Benjamin-Ono equation gives a good  and stable approximation to the system on the natural cubic time scale. 
 The proof relies on refined cubic energy estimates and perturbative analysis.
}
\vskip0.3cm

\vspace{0.5cm}



\tableofcontents

\section{Introduction}
This article continues the study of the long time dynamics
for the two dimensional water wave equation with constant vorticity, infinite depth, and gravity, but without surface tension, which was initiated by two of the authors
in \cite{ifrim2018two}. This evolution is described by the incompressible Euler equations in a moving, asymptotically flat domain $\Omega_t$, with boundary conditions on the water surface $\Gamma_t$. 

Going back to Zakharov's work in \cite{zakharov1968stability}, it has been known that, 
in the irrotational case, the water wave flow reduces 
to the study of the motion of the free boundary. A similar reduction applies in the constant vorticity case in two space dimensions \cite{wahlen2007hamiltonian,ifrim2018two}, which is the setting of this paper.

Indeed, in the appendix to \cite{ifrim2018two} it was shown that, after a transformation into holomorphic coordinates, the fluid dynamics can be reduced to a system of degenerate quasilinear hyperbolic equations governing the motion of the interface. 
It was also shown in \cite{ifrim2018two} that this system is locally well-posed in a suitable scale of Sobolev spaces, which is described in the sequel. 
Another key result of \cite{ifrim2018two} 
asserts that the constant vorticity water waves 
in deep water have a cubic lifespan. Precisely,
if the initial data has size $\epsilon$, then the lifespan of the solutions is at least $O(\epsilon^{-2})$.

Our interest in this paper is still on 
the long time dynamics of small data solutions, but this time under the additional assumption that the initial data
is localized at low frequency. 
For such solutions, a natural question is to ask whether there is a good, simplified model equation which accurately describe their behavior. Our results provide a satisfactory  answer to this question, in two steps: 

\begin{enumerate}
    \item[(i)] We identify the
Benjamin-Ono flow as the right candidate for our model equation, and 
\item[(ii)] We prove that indeed the Benjamin-Ono flow yields a good approximation for the constant vorticity 
water waves on a natural time scale.
\end{enumerate}

 Informally, our main theorem asserts that:
 
\begin{thm}
Assume that the vorticity is positive. Then the constant vorticity water waves moving to the right 
are well approximated at low frequency by  Benjamin-Ono waves on the natural cubic time scale. 

\end{thm}

Here the assumption that the vorticity is positive implies that the bulk of the nonlinear interaction at low frequency is carried by the waves moving to the right.
Full statements for our results  are provided later in 
Theorem~\ref{t:maintheorem} and Theorem~\ref{t:maintheoremWPform}, after some necessary prerequisites. 

To the best of our knowledge, this is the first time 
this problem has been considered.  Considerable 
work has been done before in the irrotational setting.
Replacing the assumption of irrotationality with constant vorticity allows the model to apply to waves in settings with countercurrents. However, the new terms introduced by the vorticity break both the scaling symmetry
and the reflection symmetry, and they have a large effect in the low-frequency regime. 

For a more precise description of our result, we 
note that the linearization of the water wave equation with constant vorticity around the zero solution is a dispersive flow, with the dispersion relation 
\begin{equation}\label{dispersion-relation}
\tau^2 + c\tau +g\xi = 0, \qquad \xi \leq 0,
\end{equation}
where $c$ is the vorticity and $\tau,\xi$ denote the 
time, respectively spatial frequencies.
The two branches have time frequency $\tau= 0$, respectively $\tau = -c$ at $\xi=0$. Of these, it is the 
first branch that carries the bulk of the nonlinear 
interactions near $\xi = 0$; this last fact is not immediately obvious, but instead it becomes clear 
only at the conclusion of this article.

The Benjamin-Ono approximation corresponds to solutions which are localized on this first branch of the dispersion relation near the zero of the frequency plane, on the frequency scale $|\xi |\lesssim \epsilon \ll 1$, and with 
$L^2$ size $\epsilon^\frac12$. Our main result asserts that, indeed, on the cubic time scale $|t| < T\epsilon^{-2}$ some solutions to the water wave equation with constant vorticity are well approximated by the appropriate Benjamin-Ono flow.
 Here T can be chosen arbitrarily large, up to $T \approx |\ln \epsilon|$, and represents the effective Benjamin-Ono time. This last fact also explains our emphasis on the cubic time scale and further on the 
 large $T$; for shorter times, the linear Benjamin-Ono flow also serves as a good approximation, even if with somewhat larger errors.

\subsection{Model problems in water waves}
Since the water wave flow is so complex and features 
a good number of parameters (gravity, capillarity, depth,
vorticity), it is natural that the question of obtaining 
good reduced model problems has received much attention
over the years. The question of justifying such approximation results is much harder, but progress has also been made in recent years.
We do not aim to provide an exhaustive 
review of such results; instead, we limit ourselves 
to the the two dimensional setting (which yields one dimensional interface evolutions) and point out 
several approximation results leading to our present 
work. But see for instance Saut's book \cite{Saut} for a broader discussion.

The universal model, which applies to all water waves
near frequencies with nondegenerate dispersion, is the 
cubic nonlinear Schr\"{o}dinger equation (NLS)
\[
i u_t +  u_{xx} = \pm u|u|^2,
\]
which arises in both its focusing and defocusing 
varieties.   In particular it was shown in \cite{ifrim2019nls} and many other previous papers starting from \cite{zakharov1968stability} that this  gives a good approximation for frequency-localized solutions to the irrotational 2D gravity water waves equations, at least on a cubic timescale.  

A second classical approximate model is the KdV flow,
\begin{equation*}\label{KdV}
u_t + u_{xxx} = 6 u u_x.
\end{equation*}
This provides an accurate description for low-frequency unidirectional waves in the context of gravity waves
with finite bottom. There the dispersion is degenerate
at frequency zero. For more details we refer the reader
to \cite{MR3363408,MR1780702,MR2196497,MR2868850}.

Finally we arrive at the Benjamin-Ono equation
\begin{equation}
\label{BenjaminOno}
    (\partial_t +H\partial_x^2)u = uu_x,
\end{equation} which was originally a model for the propagation of one dimensional internal waves, describing the physical phenomena of wave propagation at the interface of layers of fluids with different densities (see Benjamin \cite{benjamin1967internal} and Ono \cite{ono1975algebraic}, as well as a similar discussion in \cite{craig2005hamiltonian}).

The Benjamin-Ono approximation result we establish in this paper is unrelated to the one above; the only similarity is that it also describes unidirectional waves. Overall, the Benjamin-Ono equation appears to play a role which is similar to that of the KdV approximation but in settings where the dispersion remains uniform, rather than degenerate, at frequency zero. Indeed, both of them model the amplitude of the
waves at low frequency in an unidirectional setting.

As one of the referees also noted, the infinite depth of the fluid is of the essence
in our result. 
Instead, in the finite depth case one 
returns to the KdV regime at low frequency.

\subsection{2D gravity water wave equation with constant vorticity}
The water flow in the model we consider is governed by the incompressible Euler equation in the fluid domain $\Omega_t$, with a dynamic and a kinematic boundary condition on the free surface of the fluid $\Gamma_t$.
Since we are in the two dimensional case, the vorticity will also be transported along the flow. This makes it possible to study flows with a non-zero constant vorticity field.
Denoting the fluid velocity by $\mathbf{u}(t,x,y) = (u(t,x,y), v(t,x,y))$, the pressure by $p(t,x,y)$, and the constant vorticity by $c$, the equations inside $\Omega_t$ are
\begin{equation*}
\left\{
             \begin{array}{lr}
            u_t +uu_x +vu_y = -p_x &  \\
            v_t + uv_x +vv_y = -p_y -g& \\
            u_x +v_y =0 & \\
            \omega = u_y -v_x = -c,
             \end{array}
\right.
\end{equation*}
while on the boundary $\Gamma_t$ we have the dynamic boundary condition
\begin{equation*}
    p =0,
\end{equation*}
and the kinematic boundary condition
\begin{equation*}
    \partial_t +\mathbf{u}\cdot \nabla \text{ is tangent to }\Gamma_t,
\end{equation*}
where $g$ is the gravitational constant.
Introducing as a new unknown the (generalized) velocity potential $\phi(t,x,y)$, which is harmonic in $\Omega_t$ modulo an arbitrary function of time, the velocity field $\mathbf{u}$ can be represented as
\begin{equation*}
    \mathbf{u} = (cy + \phi_x, \phi_y).
\end{equation*}
We also introduce $\theta(t,x,y)$, the harmonic conjugate of $\phi(t,x,y)$, by
\begin{equation*}
    \theta_x = -\phi_y, \quad \theta_y = \phi_x.
\end{equation*}

The literature on water waves with constant vorticity is extensive; the equations are mostly formulated in  Zakharov-Craig-Sulem Eulerian formulation, see for instance \cite{vanden1996periodic}, \cite{wahlen2007hamiltonian}, \cite{constantin2008nearly}, \cite{ivanov2009two}, \cite{constantin2011steady}, \cite{escher2016two}, and \cite{berti2021traveling}. 
 Here we will instead work with the water wave equations in holomorphic (conformal) coordinates.
These coordinates are obtained via a conformal map $Z$ that maps the lower half plane $\mathcal{H}_-= \{\Im z <0 \}$,
with coordinates denoted by $z = \alpha + i\beta$, into the fluid domain $\Omega_t$. Here we require the asymptotic flatness condition
\begin{equation*}
    \lim_{\mathcal{H}_-\ni z \rightarrow \infty} Z(z) -z =0.
\end{equation*}
Removing the leading part
\begin{equation*}
    W(z) : = Z(z) -z,
\end{equation*}
we obtain our first dynamic variable W which describes the parametrization of the free boundary.
The second dynamic variable, the holomorphic velocity potential, denoted by $Q$,  is represented in terms of the real velocity potential $\phi$ and its harmonic conjugate $\theta$ as
\begin{equation*}
    Q = \phi + i\theta.
\end{equation*}

Expressed in holomorphic position/velocity potential coordinates $(W,Q)$ the water wave equations with constant vorticity have the following form:
\begin{equation}
\left\{
             \begin{array}{lr}
             W_t + (W_\alpha +1)\underline{F} +i\dfrac{c}{2}W = 0 &  \\
             Q_t - igW +\underline{F}Q_\alpha +icQ +\mathbf{P}\left[\dfrac{|Q_\alpha|^2}{J}\right]- i\frac{c}{2}T_1 =0,&  
             \end{array}
\right.\label{vorticityEqn}
\end{equation}
where $J := |1+ W_\alpha|^2$ and $\mathbf{P}$ is the projection onto negative frequencies, namely
\begin{equation*}
    \mathbf{P} = \frac{1}{2}(\mathbf{I} - i\mathbf{H}),
\end{equation*}
with $\mathbf{H}$ denoting the Hilbert transform, and
\begin{equation*}
\begin{aligned}
&F: = \mathbf{P}\left[\frac{Q_\alpha - \Bar{Q}_\alpha}{J}\right], \quad &F_1 = \mathbf{P}\left[\frac{W}{1+\Bar{W}_\alpha}+\frac{\Bar{W}}{1+W_\alpha}\right],\\
&\underline{F}: =F- i \frac{c}{2}F_1,  \quad &T_1: = \mathbf{P}\left[\frac{W\Bar{Q}_\alpha}{1+\Bar{W}_\alpha}-\frac{\Bar{W}Q_\alpha}{1+W_\alpha}\right].
\end{aligned}
\end{equation*}

\par These equations are interpreted as an evolution in the space of holomorphic functions, where, by a slight abuse of terminology, we call a function on the real line holomorphic if it admits a bounded holomorphic extension to the lower half space $\mathcal{H}$. Such functions are characterized by the property $u = \mathbf{P}u$; that is, they are restricted to negative frequencies.
For a complete derivation of the above equations we refer the reader to the appendices of \cite{ifrim2018two} and \cite{hunter2016two}.

The above equations are a Hamiltonian system, with the Hamiltonian
\begin{equation*}
    \mathcal{E}(W,Q) =\Re\int g|W|^2(1+2W_\alpha)-iQ\Bar{Q}_\alpha +cQ_\alpha(\Im W)^2 - \frac{c^3}{2i}|W|^2W(1+W_\alpha)\, d\alpha.
\end{equation*}
A second conserved quantity is the horizontal momentum
\begin{equation*}
    \mathcal{P}(W,Q) = \int \left\{\frac{1}{i}(\Bar{Q}W_\alpha -Q\Bar{W}_\alpha)-c|W|^2 + \frac{c}{2}(W^2\Bar{W}_\alpha + \Bar{W}^2W_\alpha) \right\}d\alpha,
\end{equation*}
which is the Hamiltonian for the group of horizontal translations.
\medskip

The linearization of the system~\eqref{vorticityEqn} around the zero solution is 
\begin{equation}
\left\{
             \begin{array}{lr}
            w_t +q_\alpha = 0  &\\
             q_t +icq - igw = 0,&  
             \end{array}
\right.\label{lineareqn}
\end{equation} 
viewed as an evolution on the space of holomorphic functions.
It is a linear dispersive equation with the dispersion relation
\begin{equation}
    \tau^2 +c\tau +g\xi = 0, \quad \xi \leq 0, \label{dispersive}
\end{equation} 
depicted in Figure~\ref{dr}. Here to fix the geometry 
we have assumed that $c > 0$. 

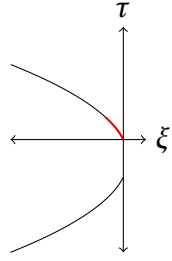
\begin{figure}[!htb]
\begin{tikzpicture}
\draw[smooth,samples=100,domain=0:1,variable=\y] plot({-\y*\y-0.5*\y},{\y});
\draw[smooth,thick,red,samples=100,domain=0:0.3,variable=\y] plot({-\y*\y-0.5*\y},{\y});
\draw[smooth,samples=100,domain=-1.5:-0.5,variable=\y] plot({-\y*\y-0.5*\y},{\y});
\draw [<->] (-1.5,0) -- (0.3,0) node[anchor=west]{$\xi$};
\draw [<->] (0,-1.5) -- (0,1.5) node [anchor=south]{$\tau$};
\end{tikzpicture}
\caption{The dispersion relation for gravity water waves with constant vorticity; the red region is the low-frequency regime we are interested in.}\label{dr}
\end{figure}

The conserved energy, respectively momentum for $\eqref{lineareqn}$ are
\begin{equation*}
\begin{aligned}
&\mathcal{E}_0(w,q) = \int |w|^2 -iq\Bar{q}_\alpha d\alpha = \|w\|^2_{L^2}+\|q\|^2_{\dot{H}^{\frac{1}{2}}},\\
&\mathcal{P}_0(w,q) = \int \frac{1}{i}(\Bar{q}w_\alpha -q\Bar{w}_\alpha)-c|w|^2\, d\alpha.
\end{aligned}    
\end{equation*}
The energy space described by the above conserved energy is denoted by 
\[
\mathcal H = L^2 \times \dot H^\frac12.
\]

 To measure higher regularity we will use the homogeneous spaces $\doth{n}$ endowed with the norm 
 \begin{equation*}
    \|(w,r)\|^2_{\doth{n}} : = \|\partial_\alpha^n (w,r) \|^2_{L^2\times \Dot{H}^{\frac{1}{2}}},
 \end{equation*} 
and the inhomogeneous space $\mathcal H^{n}$ with the norm 
\begin{equation*}
    \|(w,r)\|^2_{\mathcal H^{n}} : = \sum_{k=0}^n \|\partial_\alpha^k (w,r) \|^2_{L^2\times \Dot{H}^{\frac{1}{2}}}.
 \end{equation*}
 
 The equations \eqref{vorticityEqn} are fully nonlinear.
 We can differentiate them to convert them into a quasilinear form.  For the differentiated equations, it 
 also useful to diagonalize.  As in \cite{ifrim2018two}, we define the operator
 \begin{equation}\label{bfA-def}
    \mathbf{A}(w,q) : = (w, q-Rw),\quad R:=\frac{Q_\alpha}{1+W_\alpha}.
 \end{equation}
 Then the diagonal variables for the differentiated equation are 
 \begin{equation}
(\mathbf{W}, R) : = \mathbf{A}(\mathbf{W}, Q_\alpha), \qquad
\mathbf{W} := W_\alpha.
 \end{equation}
 then we arrive at the system
 \begin{equation}
\left\{
             \begin{array}{lr}
            \mathbf{W}_t +\underline{b}\mathbf{W}_\alpha + \dfrac{(1+\mathbf{W}R_\alpha)}{1+\Bar{\mathbf{W}}} = (1+\mathbf{W})\underline{M}+i\dfrac{c}{2}\mathbf{W}(\mathbf{W}-\Bar{\mathbf{W}})  &\\
             R_t + \underline{b}R_\alpha +icR - i\dfrac{g\mathbf{W}-a}{1+\mathbf{W}} = i\dfrac{c}{2}\dfrac{R\mathbf{W}+\Bar{R}\mathbf{W
             +N}}{1+\mathbf{W}}.&  
             \end{array}
\right.\label{differentiatedEqn}
\end{equation}
Here $\underline{b}$ is the \textit{advection velocity}, which is given by
\begin{equation}
\underline{b}: = b-i\frac{c}{2}b_1,\quad b:= \mathbf{P}\left[\frac{Q_\alpha}{J}\right]+\Bar{\mathbf{P}}\left[\frac{\Bar{Q}_\alpha}{J}\right], \quad b_1: = \mathbf{P}\left[\frac{W}{1+\Bar{\mathbf{W}}}\right]-\Bar{\mathbf{P}}\left[\frac{\Bar{W}}{1+\mathbf{W}}\right].
\end{equation}
The \textit{frequency-shift} $a$ is given by
\begin{equation}
\underline{a} := a +\frac{c}{2}a_1, \quad a: = i(\Bar{\mathbf{P}}[\Bar{R}R_\alpha]- \mathbf{P}[R\Bar{R}_\alpha]), \quad a_1 = R + \Bar{R} -N,
\end{equation}
 where
 \begin{equation}
     \mathbf N: = \mathbf{P}[W\Bar{R}_\alpha - \Bar{\mathbf{W}}R]+\Bar{\mathbf{P}}[\Bar{W}R_\alpha - \mathbf{W}\Bar{R}].
 \end{equation}
 The expression $g+ \underline a$ plays an essential role, and 
 represents the normal derivative of the pressure on the free 
 boundary.

\subsection{An overview of the results of \texorpdfstring{\cite{ifrim2018two}}{[12]}} 
To provide some context for our Benjamin-Ono approximation results in the next subsection, we need to briefly review
the results of \cite{ifrim2018two}.
The first main theorem of \cite{ifrim2018two} is the local well-posedness result,
\begin{theorem}[\hspace{1sp}\cite{ifrim2018two}]
 Let $n\geq 1$. The system $\eqref{vorticityEqn}$ is locally well-posed for initial data $(W_0, Q_0)$ with the following regularity:
 \begin{equation*}
     (W_0, Q_0)\in \mathcal H, \quad (\mathbf{W}_0, R_0) \in \doth{1},
 \end{equation*}
 and which satisfy the pointwise bounds
 \begin{align*}
   |\mathbf{W_0}(\alpha)+1|>\delta>0&  \quad\text{(no interface singularities)}, \\
    g+ \underline{a_0}(\alpha) >\delta>0& \quad\text{(Taylor sign condition)}.
 \end{align*}

\end{theorem}

Here well-posedness should be understood in a standard quasilinear fashion,i.e. with the solutions remaining in the same space as the initial data but only depending continuously on the initial data.
Higher $\mathcal H^m$ regularity for the above solutions is also 
shown to be propagated along the flow. The Taylor sign condition is 
necessary for local well-posedness, in that the problem becomes ill-posed if the sign of the normal derivative of the pressure is reversed~\cite{MR36104},~\cite{MR1231428}.

The second result in \cite{ifrim2018two} is the long time result, giving a cubic lifespan for small data solutions:
\begin{theorem}[\hspace{1sp}\cite{ifrim2018two}]\label{t:cubictimeandhigherregularity}
  Let $(W,Q)$ be a solution for the system $\eqref{vorticityEqn}$ whose data satisfies
  \begin{equation}
      \|(W_0,Q_0)\|_{\mathcal H}+\|(\mathbf{W}_0, R_0)\|_{\doth{1}}\leq \epsilon \ll 1.
  \end{equation}
  Then the solution exists for a time $T_\epsilon \approx \epsilon^{-2}$, with bounds
  \begin{equation}
      \| (W,Q)(t)\|_{\mathcal H}+\|(\mathbf{W}, R)(t)\|_{\doth{1}}\lesssim \epsilon, \quad |t|\leq T_\epsilon.
  \end{equation}
  Further, higher regularity is also preserved,
  \begin{equation}
      \| (\mathbf{W},R)(t)\|_{\doth{n}}\lesssim \|(\mathbf{W},R)(0)\|_{\doth{n}}, \quad |t|\leq T_\epsilon,
  \end{equation}
  whenever the norm on the right is finite.
\end{theorem}

The proof of the above cubic lifespan result is based on the cubic energy estimates of \cite{ifrim2018two}. To state them  we define the control norms
\begin{equation*}
\begin{aligned}
A : &= \|\mathbf{W} \|_{L^\infty}+\|Y\|_{L^\infty} +\||D|^{\frac{1}{2}}R \|_{L^\infty \cap B^0_{\infty, 2}},\\
B : &= \||D|^{\frac{1}{2}}\mathbf{W} \|_{BMO}+\|R_\alpha \|_{BMO},\\
A_{-\frac{1}{2}} : &= \| |D|^{\frac{1}{2}}W\|_{L^\infty} + \|R\|_{L^\infty},\quad A_{-1}:= \|W\|_{L^\infty},\\
\underline{B}: &=B+cA +C^2A_{-\frac{1}{2}}, \quad \underline{A}: = A+cA_{-\frac{1}{2}}+c^2A_{-1}.
\end{aligned}
\end{equation*}
Then according to \cite{ifrim2018two} we have the following energy estimates in the small data regime:
\begin{proposition}[\hspace{1sp}\cite{ifrim2018two}]\label{t:oldcubicenergy}
 For any $n\geq 0$ there exists an energy functional $E^{n,(3)}$ which has the following properties as along as $\underline{A}\ll 1$:
 \begin{enumerate} [label=(\roman*)]
     \item Norm equivalence: 
       \begin{equation}
           E^{n,(3)}(\mathbf{W},R) = (1+O(\underline{A}))\mathcal{E}_0(\partial^n \mathbf{W}, \partial^n R) + O(c^4 \underline{A})\mathcal{E}_0(\partial^{n-1}\mathbf{W}, \partial^{n-1}R).\label{normEquivalence}
       \end{equation}
     \item Cubic energy estimates:
      \begin{equation}
          \frac{d}{dt}E^{n,(3)}(\mathbf{W},R) \lesssim_{\underline{A}}\underline{B}\underline{A}\left(\mathcal{E}_0(\partial^n \mathbf{W}, \partial^n R)+c^4\mathcal{E}_0(\partial^{n-1}\mathbf{W}, \partial^{n-1}R)\right).\label{oldEnergyEstimates}
      \end{equation}
      Here if $n=0$ then $\mathcal{E}_0(\partial^{-1}\mathbf{W}, \partial^{-1}R)$ is naturally replaced by $\mathcal{E}(W,Q)$.
 \end{enumerate}
\end{proposition}

An important role will also be played by the
linearized water wave equation with constant vorticity projected onto holomorphic space. Here we denote 
the linearized variables by $(w,q)$, but then we diagonalize
using the operator $\mathbf A$ in \eqref{bfA-def}, and write the equations in terms of the good variables 
\[
(w,r) = \mathbf A(w,q). 
\]
For $(w,r)$ we have the linearized system 
\begin{equation}
\left\{
             \begin{array}{lr}
             (\partial_t + \mathfrak{M}_{\underline{b}}\partial_\alpha)w +\mathbf{P}\left[\dfrac{1}{1+\Bar{\mathbf{W}}}r_\alpha\right]+\mathbf{P}\left[\dfrac{R_\alpha}{1+\Bar{\mathbf{W}}}w\right] = \mathbf{P}\underline{\mathcal{G}}(w,r)&\\
            (\partial_t + \mathfrak{M}_{\underline{b}}\partial_\alpha)r +icr-i\mathbf{P}\left[\dfrac{g+\underline{a}}{1+\mathbf{W}}w\right]=\mathbf{P}\underline{\mathcal{K}}(w,r),&  \label{linearizedeqn}
             \end{array}
\right.
\end{equation}
where the source terms $\underline{\mathcal{G}}(w,r)$ and $\underline{\mathcal{K}}(w,r)$ are given by
\begin{equation*}
\begin{aligned}
\underline{\mathcal{G}}(w,r) = \mathcal{G}(w,r)-i\frac{c}{2}\mathcal{G}_1(w,r),& \quad \underline{\mathcal{K}}(w,r) = \mathcal{K}(w,r)-i\frac{c}{2}\mathcal{K}_1(w,r), \\
\mathcal{G}(w,r) = (1+\mathbf{W})(\mathbf{P}\Bar{m}+\Bar{\mathbf{P}}m),& \quad \mathcal{G}_1(w,r) = -(1+\mathbf{W})(\mathbf{P}\Bar{m}_1-\Bar{\mathbf{P}}m_1)+\frac{(\Bar{\mathbf{W}}-\mathbf{W})w}{1+\Bar{W}},\\
\mathcal{K}(w,r) = \Bar{\mathbf{P}}n - \mathbf{P}\Bar{n},& \quad \mathcal{K}_1(w,r) = \Bar{\mathbf{P}}m_2 + \mathbf{P}\Bar{m}_2, \quad n : = \frac{\Bar{R}(r_\alpha +R_\alpha w)}{1+\mathbf{W}},
\end{aligned}
\end{equation*}
as well as
\begin{equation*}
    \begin{aligned}
    &m : = \frac{q_\alpha - Rw_\alpha}{J}+\frac{\Bar{R}w_\alpha}{(1+\mathbf{W}^2)} = \frac{r_\alpha + R_\alpha w}{J}+\frac{\Bar{R}w_\alpha}{(1+\mathbf{W}^2)},\\
    &m_1: = \frac{1}{1+\Bar{\mathbf{W}}}w - \frac{\Bar{W}}{(1+W)^2}w_\alpha, \quad m_2 := \Bar{R}w - \frac{\Bar{W}r_\alpha +\Bar{W}R_\alpha w}{1+\mathbf{W}}.
    \end{aligned}
\end{equation*}
We can define the associated (nearly) cubic linearized energy
 \begin{equation}
    E^{(3)}_{lin}(w,r)=\int (g+\underline a)|w|^2+\Im(r\overline r_\alpha)+2\Im (Rwr_\alpha)-2\Re(\overline{\W}w^2)\, d\alpha.\label{lincubicenergy}
\end{equation}
Then the main energy estimate for the linearized equation has the form

\begin{proposition}[\hspace{1sp}\cite{ifrim2018two}]\label{t:linearizedestimates}
 Assume that $\underline{A}$ is small, then the linearized equation is well-posed in $L^2\times \dot{H}^{\frac{1}{2}}$, and the following properties hold:
 \begin{enumerate}[label=(\roman*)]
     \item Norm equivalence: 
       \begin{equation*}
           E^{(3)}_{lin}(w,r) \approx_{\underline{A}} \mathcal{E}_0(w,r).
       \end{equation*}
     \item Energy estimates:
      \begin{equation*}
     \frac{d}{dt}E^{(3)}_{lin}(w,r)=2\Re \int c\left[gw\nP\mathcal G_1^{(2)}(w,r)-i\overline r_\alpha\nP\mathcal K_1^{(2)}(w,r)\right]-\frac{ic^2}{2}R|w|^2d\alpha+O_A(\uA\uB)E_{lin}^{(3)}(w,r),
 \end{equation*}
  \end{enumerate}where 
 \begin{equation*}
 \begin{aligned}
 &\nP\mathcal G_1^{(2)}(w,r) = \nP[\W \Bar{w}] + \nP[W \Bar{w}_\alpha] +\nP[{\Bar{\mathbf{W}}w}] - \nP[\mathbf{W}w], \\
& \nP\mathcal K_1^{(2)}(w,r) = \nP[{W\Bar{r}_\alpha}] - \mathbf{P}[R\Bar{w}].
 \end{aligned}
 \end{equation*}

\end{proposition} \label{CubicLinearEnergy}

\subsection{The Benjamin-Ono approximation}
To describe the Benjamin-Ono flow associated to the 
constant vorticity water wave equations,  we begin with the Benjamin-Ono equation
\[
u_t + H \partial_x^2 u =  u u_x.
\]
We remark that solutions $u$ are  real valued  functions, while the solutions to the water wave flow are holomorphic functions; in other words, they have negative spectrum. To account for this difference, we 
observe that the Benjamin-Ono flow can be equivalently recast as  an evolution for its projection $z = - i\mathbf P u$ to negative frequencies, namely 
\begin{equation}\label{proj-bo}
(i\partial_t  + \partial_x^2) z =  - z z_x + \textbf P [z \bar z_x + \bar z z_x].     
\end{equation}
This evolution is somewhat similar to, but not identical  
with, the derivative NLS flow.

Here $u$ can be recovered by $u = -2 \Im z$. It is this form of the Benjamin-Ono equation which will serve as 
the approximation for water waves.
\bigskip

To construct our approximation, we begin with the linear
analysis, starting from the dispersion relation 
\eqref{dispersive} for the linear flow \eqref{lineareqn}. 
As a polynomial in $\tau$, 
for $\xi < 0$ this has two real roots, 
\[
\omega^{\pm}(\xi) = - \frac12\left( c \mp \sqrt{c^2+4g|\xi|} \right)
\]
which correspond to the two branches in Figure~\ref{dr}.
We are interested in low-frequency solutions along the first branch; solutions $(w,q)$ to the linearized equation which are localized there will  solve
\begin{equation}\label{first-branch}
(i \partial_t + \omega^+(D))(w,q) = 0, \qquad g w + \omega^{-}(D) q = 0.
\end{equation}
Here it is critical that the second derivative of the symbol $\omega^+$
is nonzero at $\xi =0$, which is what leads to the Benjamin-Ono approximation, and which is also stable with respect to small smooth perturbations. 
Then, as in the paper\cite{ifrim2019nls}, we take the quadratic approximation of the dispersion relation near $(\xi_0, \tau_0) = (0,0)$, which is 
\begin{equation*}
    \tau = \omega_0(\xi) : = -\frac{g}{c}\xi -\frac{g^2}{c^3}\xi^2,
\end{equation*}
and which corresponds to the linear evolution
\begin{equation}
    (i\partial_t +\omega_0(D))y = 0.
\end{equation}
This is a linear Schr\"odinger equation with an added 
$g/c$ velocity, which is exactly the group velocity 
for frequency zero waves in \eqref{first-branch}. In other
words, we can rewrite this as 
\begin{equation}
    (i\partial_t +\frac{g^2}{c^3}\partial_x^2)z = 0
\end{equation}
via the transformation
\begin{equation*}
    z(t,x) = y(t, x- \frac{g}{c}t).
\end{equation*}
We will interpret this last equation as the linear part 
of the projected Benjamin-Ono flow \eqref{proj-bo}, with
an extra scaling factor added.

This means that for the linearized water wave system with constant vorticity \eqref{vorticityEqn} on the upper branch, i.e. as in \eqref{first-branch},
we can approximate the solution $(w,q)$ by $(y,\frac{g}{c} y)$, at least at very low frequencies.
For a more quantitative analysis, suppose that we are looking at solutions concentrated at frequencies in the range $(-\epsilon,0]$.
We then have the difference relation
\begin{equation*}
|\omega_0 - \omega_{+}|\lesssim \epsilon^3.    
\end{equation*}
Thus the linear evolution operator $e^{it\omega_{+}(D)}$ is well-approximated by the linear Benjamin-Ono flow $e^{it\omega_0(D)}$ on the time scale
\begin{equation*}
    |t| \ll \epsilon^{-3}.
\end{equation*}

\bigskip

Now we consider the nonlinear setting. We begin with a solution $U$ for the Benjamin-Ono equation
\begin{equation}
    (\partial_t -\frac{g^2}{c^3}H\partial^2_x)U = \lambda UU_x, \qquad U(0) = U_0,
    \label{BenjaminOnoEqn}
\end{equation}
where the additional scaling factor $\lambda$ remains to be chosen later. We rescale this to the $\epsilon$ frequency scale
\begin{equation}
    U^{\epsilon}(t,x) = \epsilon U(\epsilon^2 t, \epsilon x), \label{TransformOne}
\end{equation}
which still solves the above Benjamin-Ono equation.

Next, we define its projection $Z^\epsilon = - i \mathbf P U^\epsilon$. This solves \eqref{proj-bo} but with extra scaling factors,
\begin{equation}
     (i\partial_t+\frac{g^2}{c^3}\partial^2_x)Z^\epsilon=\lambda\left(- Z^\epsilon Z^\epsilon_x+\nP\left[Z^\epsilon \overline{Z^\epsilon_x}\right]+\nP\left[\overline{Z^\epsilon} Z^\epsilon_x\right]\right).\label{TransformTwo}
\end{equation}
Finally, we add in the desired group velocity associated
to the zero frequency waves, and set  
\begin{equation}
    Y^\epsilon(t,x) := Z^\epsilon(t, x - \frac{g}{c}t). \label{TransformThree}
\end{equation}
This in turn solves 
\begin{equation}
     (i\partial_t+\omega_0(D))Y^\epsilon=\lambda\left(- Y^\epsilon Y^\epsilon_x+\nP\left[Y^\epsilon \overline{Y^\epsilon_x}\right]+\nP\left[\overline{Y^\epsilon} Y^\epsilon_x\right]\right).\label{negativeFrequencyBO}
\end{equation}
The function $Y^\epsilon$ will be our initial candidate 
for an approximate solution to the constant vorticity
water wave flow, though we will have to apply a normal form transformation in order to sufficiently improve the accuracy. Our results will go both ways:
\begin{enumerate}[label=(\roman*)]
    \item For each approximate Benjamin-Ono type solution $Y^\epsilon$ we can find an exact constant vorticity water wave flow nearby.
   \item Each $\epsilon$ localized water wave with well-prepared initial data can be approximated by a Benjamin-Ono type solution $Y_\epsilon$ as above.
\end{enumerate}

 Our first main theorem covers part (i) above:
\begin{theorem}\label{t:maintheorem}
Let $m\ge 3$ and $T > 0$.
 Let $U_0 \in H^m$, let $U$ be the corresponding solution of Benjamin-Ono equation \eqref{BenjaminOnoEqn} with $\lambda = c$ in $[0,T]$, and $Y^\epsilon$ be as in \eqref{negativeFrequencyBO}. Then there exists $0<  \epsilon_0(\|U\|_{H^m},T)$  so that for each $0<\epsilon<\epsilon_0$ there exists a solution $(W,Q)$ to \eqref{vorticityEqn} for $t$ in the time interval $\{ |t|\leq T_\epsilon := T\epsilon^{-2} \}$ satisfying the following estimates
 \begin{equation}
        \left\|\left(W-2Y^\epsilon,Q-\frac{2g}{c}Y^\epsilon\right)\right\|_{\mathcal H}\lesssim _T \epsilon^{\frac 3 2} \label{mainresult}
 \end{equation}
 and 
 \begin{equation}
     \left\|\left(\W,R\right)\right\|_{\doth{n}}\lesssim _T \epsilon^{\frac{3}{2}+n}\label{mainresulthf}
 \end{equation} for $0\le n\le m-1$.
\end{theorem}

\begin{remark}
It is clear from the set-up preceding the theorem that 
the Benjamin-Ono flow models the low frequency evolution of $\Im W$, which describes the amplitude
of the waves in the holomorphic parametrization.
\end{remark}

\begin{remark}
The aim of the notation $\lesssim_T$ in this theorem and also in subsequent theorems is to emphasize the fact that the implicit constant in our estimates is on one hand independent of $\epsilon$, but on the other hand it does depend on $T$. As it turns out in the proof, the $T$ dependence is exponential in the difference bounds \eqref{mainresult}, but only polynomial in 
the energy estimates in \eqref{mainresulthf}. 
\end{remark}

\begin{remark}\label{r:epsilon0}
By construction, the water wave solutions above have initial data size 
\[
\|(W_0,Q_0)\|_{\mathcal H} \lesssim \epsilon^\frac12,
\]
which is much larger than the solutions in Theorem~\ref{t:cubictimeandhigherregularity}. 
On the other hand, due to the $\epsilon$ frequency localization,  higher regularity bounds are stronger here, e.g. 
\[
\|(\W_0,R_0)\|_{\mathcal H} \lesssim \epsilon^\frac32, \qquad 
\|(\W_0,R_0)\|_{\doth{1}} \lesssim \epsilon^\frac52.
\]
However, ultimately, in the energy estimates, it is the 
pointwise bound for $\uA \uB$ which matters, and this 
is of size $\epsilon^2$ in both Theorem~\ref{t:cubictimeandhigherregularity}
and in the theorem above.
The conclusion~\eqref{mainresulthf} guarantees that the solution satisfies higher regularity bounds that resemble continued $\epsilon$ frequency localization.

\end{remark}

\begin{remark}
One may also obtain $L^\infty$ error bounds 
by interpolating between the two estimates in the theorem and then by using Sobolev embeddings. However,
these bounds are not optimal since the sharp difference estimate \eqref{mainresult} is only carried out directly at the lowest energy level in this paper.
Likely this can be improved, but we have chosen 
not to pursue it.
\end{remark}

\begin{remark}
It is important in the above result that the effective 
Benjamin-Ono time is large enough. Indeed, for small $T$, one might equally use the linear Benjamin-Ono flow instead as a good approximation.
In particular, it is essential to work with the 
water wave equations on the cubic time scale $\epsilon^{-2}$. This time scale also corresponds 
to taking the (nondegenerate) quadratic approximation to the dispersion relation near frequency zero.

A careful analysis of the proof of the theorem reveals that the $\epsilon_0$ in the theorem may be chosen as 
\[
\epsilon_0 = e^{-CT},
\]
with a large universal constant $C$.
Turning the tables it follows that the lifespan of the 
solutions in the theorem is 
\[
T_\epsilon \approx \epsilon^{-2} |\ln \epsilon|.
\]
This is logarithmically 
larger than what a result as in Theorem~\ref{t:cubictimeandhigherregularity} would yield.

\end{remark}

For our second result, we cannot work with arbitrary water waves, since the Benjamin-Ono approximation should only apply to those which 
correspond to the upper branch of the dispersion relation.
Unlike the linear case, in the nonlinear problem we cannot
impose an exact condition.  Instead, we observe that, 
diagonalizing the linearized system~\eqref{lineareqn}, the subspace of $L^2\times\dot H^{\frac 1 2}$ corresponding to the top branch of the dispersion relation~\eqref{dispersive} consists of those pairs $(W,Q)$ with 
\[
W=\frac{c+\sqrt{c^2+4g|D|}}{2g}Q.
\]
Working with functions $(W,Q)$ of $\mathcal H$ norm $\epsilon^{\frac 1 2}$ which are concentrated at frequencies $\lesssim \epsilon$, this implies that 
\[
Q= \frac{g}{c}W+O_{\dot H^\frac12}(\epsilon^{\frac 3 2}).
\]
This last condition is weaker but more robust, sufficiently so that we can transfer it to the nonlinear flow, which allows for interaction between the two branches. Precisely, it allows us to introduce the following notion of well prepared data:

\begin{definition}
Let $m \geq 4$. An initial data $(W_0,Q_0)\in \mathcal H^m$ for the system \eqref{vorticityEqn} is said to be \emph{$\epsilon$-well-prepared in $\mathcal H^m$} if it has the following two properties: \begin{enumerate}
    \item Frequency concentration on the $\epsilon$ scale: $W_0$ and $Q_0$ satisfy the following estimates: 
    \begin{equation}
\|(W_0,Q_0)\|_{\mathcal H} \lesssim \epsilon^\frac12,
\qquad 
  \|\partial^j (\W_0,R_0)\|_{\mathcal H} \lesssim \epsilon^{j+\frac 3 2}, \quad 0 \leq j \leq m-1.
     \label{e:epsilonfrequencyconcentration+}
    \end{equation}

    \item Coupling of velocities: $W_0$ and $Q_0$ have the following relationship: 
    \begin{equation} \label{match}
    W_0-\frac{c}{g}Q_0=O_{\dot H^\frac12}(\epsilon^{\frac 3 2}).
\end{equation}
\end{enumerate}
\end{definition}
We note that in Eulerian coordinates, one may simply replace $W_0$ and $Q_0$ in \eqref{match}
by the amplitude of the free surface, respectively the trace of the velocity 
potential on the free surface.

Then for the $\epsilon$-well-prepared initial data, we have the following theorem.

\begin{theorem}\label{t:maintheoremWPform}
Let $m\ge 3$ and $T>0$. 
Then for all sufficiently small $\epsilon$ and all $\epsilon$-well-prepared initial data $(W_0,Q_0)$ in $\mathcal H^m$ for~\eqref{vorticityEqn}, the solution $(W,Q)$ exists on $[0,T \epsilon^{-2}]$, and there exists a solution $U$ of the Benjamin-Ono 
equation \eqref{BenjaminOnoEqn}, with  $Y^\epsilon$ the associated solution to \eqref{negativeFrequencyBO},
satisfying the following estimates within the same time interval: 
\begin{equation}
 \|U_0\|_{H^{m}} \lesssim 1,  
\end{equation}
\begin{equation}
        \left\|\left(W-2Y^\epsilon,Q-\frac{2g}{c}Y^\epsilon\right)\right\|_{\mathcal{H}}\lesssim_T \epsilon^{\frac 3 2},
 \end{equation} and \begin{equation}
     \left\|\left(\W,R\right)\right\|_{\doth{n}}\lesssim_T\epsilon^{\frac{3}{2}+n}
 \end{equation} for $0\le n\le m-1$.
\end{theorem}
\begin{remark}
As in Theorem~\ref{t:maintheorem}, here $\epsilon$ needs to be exponentially small, $\epsilon \leq e^{-CT}$. Alternatively, one may turn the tables and fix $\epsilon$, in which case $T$ can be allowed to be as large as $|\log \epsilon|$.
\end{remark}

\subsection{An outline of the proof}

There are three major ingredients in our proof:
\begin{enumerate}[label=(\roman*)]
    \item A normal form type decoupling analysis for the system \eqref{vorticityEqn}, in order 
to nonlinearly separate at leading order the left and right going waves,
    \item Refined cubic energy estimates improving (at least in our setting) those from \cite{ifrim2018two}, and
    \item Perturbative analysis based on the linearized equation.
\end{enumerate}

These are developed in several steps, which are briefly described in what follows, together with some modular intermediate results:

\bigskip

\emph{I. Benjamin-Ono truncation.} The solution $U$ for Benjamin-Ono has all frequencies in it. However, for large enough
frequencies ($\gtrsim \epsilon^{-1}$), the dispersion relations for 
Benjamin-Ono and water waves are far apart, and there cannot be any  
good correspondence between solutions. For this reason, in Section \ref{s:BOtruncation}, we replace the exact solution $U$ to the Benjamin-Ono equation with a more regular function $\tilde{U}$, with an $\epsilon$ dependent truncation scale.
The function $\tilde U$ is only an approximate solution to the Benjamin-Ono equation, but the error introduced by the truncation can be estimated and does not pose a problem for our analysis.
We then define the corresponding approximate negative frequency Benjamin-Ono solution $\tilde{Y}^\epsilon$ and state the key estimates it satisfies.

\bigskip
\emph{II. Approximate water waves.}
The second step, which is carried out in Section \ref{s:BOapproximate}, is to start from $Y_\epsilon$ and  construct an approximate solution to the water wave equation with constant vorticity which is close enough to $(2\tilde{Y}^\epsilon, \frac{2g}{c}\tilde{Y}^\epsilon)$. The conclusion of this step can be stated  as follows:

\begin{theorem}\label{t:ww-approx}
 Let $0 < \epsilon \ll 1$. Let $\tilde Y^\epsilon$ be the $H^m$ approximate solution which will be defined in Section \ref{s:BOtruncation} to the projected Benjamin-Ono equation \eqref{negativeFrequencyBO}. 
 Then there exists a frequency localized approximate solution $(W^\epsilon, Q^\epsilon)$ for \eqref{vorticityEqn} with properties as follows:
 \begin{enumerate} [label=(\roman*)]
     \item Uniform bounds:
     \begin{equation}
         \left\|\left(W^\epsilon-2\tilde Y^\epsilon,Q^\epsilon-\frac{2g}{c}\tilde Y^\epsilon\right)\right\|_{\mathcal H}\lesssim \epsilon^{\frac 3 2}.\label{e:firststepestimates}
     \end{equation}
     \item $\epsilon$-frequency concentration:
     \begin{equation}
         \|(\W^\epsilon,R^\epsilon)\|_{\doth{n}}\lesssim \epsilon^{\frac 3 2+n}\label{e:firststepfreqconcentration}
     \end{equation} for $0\le n\le m-1$.

     \item Approximate solution:
     \begin{equation}
\left\{
             \begin{array}{lr}
             W^\epsilon_t + (W^\epsilon_\alpha +1)\underline{F}^\epsilon +i\frac{c}{2}W^\epsilon = g^\epsilon &  \\
             Q_t^\epsilon - igW^\epsilon +\underline{F}^\epsilon Q^\epsilon_\alpha +icQ^\epsilon +P\left[\frac{|Q^\epsilon_\alpha|^2}{J^\epsilon}\right]- i\frac{c}{2}T^\epsilon_1 = k^\epsilon,&  
             \end{array}
\right.
\end{equation}
where $(g^\epsilon, k^\epsilon)$ satisfy the bounds 
\begin{equation}\label{gk-epsilon}
    \|g^\epsilon\|_{H^1}+ \|k^\epsilon \|_{\dot H^{\frac 1 2}}\lesssim \epsilon^{\frac{7}{2}}.
\end{equation}
 \end{enumerate} \label{t:section4theorem}
\end{theorem}
We remark that, consistent with our energy norm $\mathcal H$, in \eqref{gk-epsilon} one might naturally expect to have only the $L^2$ norm of $g^\epsilon$.
The (technical) reason why the full $H^1$ norm is present here is in order 
to facilitate the transition to the good variables in the proof of Theorem~\ref{t:exact} below, where $(g^\epsilon,k^\epsilon)$ are interpreted as initial data for a linearized equation.

\bigskip

\emph{III. Improved long time energy estimates.}
The cubic energy estimates of \cite{ifrim2018two}, while on the correct time scale, are adapted to the unit frequency scale and not the $\epsilon$ frequency scale. For this reason, we cannot use them directly to transport the $\epsilon$ frequency concentration along the water wave flow. In Section~\ref{s:refinedenergy} we rectify this, and prove 
a refined energy estimate that can be used to better propagate the $\epsilon$ frequency concentration.

\bigskip

\emph{IV. Exact water waves.}
The next part of the proof is to show that the approximate solution $(W^\epsilon, Q^\epsilon)$ provided by Theorem~\ref{t:ww-approx}
can be upgraded to an exact solution. This is achieved 
via energy estimates for both the exact solutions and for the linearized equation.

\begin{theorem}\label{t:exact}
 Let $T>0$, and $(W^\epsilon, Q^\epsilon)$ be the approximate solution for the water wave equation with constant vorticity \eqref{vorticityEqn} given by the previous theorem.
 Then for $\epsilon$ small enough (depending only on $\| U_0\|_{H^m}$ and on $T$), the exact solution $(W,Q)$ with the same initial data exists in the time interval $\{ |t|\leq T\epsilon^{-2}\}$, and satisfies the bounds
 \begin{equation}
     \|(W- W^\epsilon, Q-Q^\epsilon) \|_{\mathcal{H}}\lesssim_T \epsilon^{\frac{3}{2}}\label{e:secondstepestimates}
 \end{equation} and \begin{equation}
     \left\|\left(\W,R\right)\right\|_{\doth{n}}\lesssim_T \epsilon^{\frac{3}{2}+n}
 \end{equation} for $0\le n\le m-1$.\label{t:exactsolutionclose}
\end{theorem}

\bigskip

\emph{V. Perturbation theory.}
Finally, we show in the last section that the Benjamin-Ono approximation given by Theorem \ref{t:maintheorem} are stable with respect to small perturbations, as long as those perturbations are frequency-concentrated near frequency zero. 
This in turn implies our second main result in Theorem~\ref{t:maintheoremWPform}. 
The result here can be stated as follows: 

\begin{theorem}\label{t:perturbation}
Let $m \geq 3$.
Let $U_0, U, Y^\epsilon$ be as in Theorem \ref{t:maintheorem}. 
Let $(W_0, Q_0)\in \mathcal H^{m}$ be initial data satisfying the estimates 
\begin{equation}
    \|(W_0,Q_0)\|_{\mathcal H} \lesssim \epsilon^\frac12,
    \qquad 
  \|\partial^j (\W_0,R_0)\|_{\mathcal H} \lesssim \epsilon^{j+\frac 3 2}, \quad 0 \leq j \leq m-1
     \label{e:perturbationfrequencyconcentration}
\end{equation} 
which correspond to frequency concentration on the $\epsilon$ scale (these are also the frequency concentration estimates for $\epsilon$-well-prepared initial data in $\mathcal H^{m}$), $\frac{1}{2m-2}<\delta \le \frac{1}{2}$, and
 \begin{equation}
        \left\|\left(W_0-2Y^\epsilon_0,Q_0-\frac{2g}{c}Y^\epsilon_0\right)\right\|_{\mathcal{H}}\lesssim \epsilon^{1+\delta}.\label{e:perturbinitial}
\end{equation}

Then the corresponding solution $(W,Q)$ to system \eqref{vorticityEqn} exists for $t$ in the time interval $\{|t|\leq T_\epsilon := T\epsilon^{-2} \}$ and satisfies the following estimates:
\begin{equation}
       \left\|\left(W-2Y^\epsilon,Q-\frac{2g}{c}Y^\epsilon\right)\right\|_{\mathcal H}\lesssim_T \epsilon^{1+\delta}.\label{e:lfperturbationbound}
\end{equation} and \begin{equation}
    \left\|\left(\W,R\right)\right\|_{\doth{n}}\le _T \epsilon^{\frac{3}{2}+n}\label{e:hfperturbationbound}
\end{equation} for $0\le n\le m-1$.
\end{theorem}

\textbf{Acknowledgments.}
Mihaela Ifrim was supported by a Clare Boothe Luce Professorship, by the Sloan Foundation, and by an NSF CAREER grant DMS-1845037. 
Daniel Tataru was supported by the NSF grant DMS-2054975 as well as by a Simons Investigator grant from the Simons Foundation.  
James Rowan was partially supported by the Raymond H. Sciobereti Fellowship and also by the same
Simons grant for Summer 2021.
This material is also based upon work supported by the National Science Foundation under Grant No. DMS-1928930 while all four authors participated in the program \textit{Mathematical problems in fluid dynamics} hosted by the Mathematical Sciences Research Institute in Berkeley, California, during the Spring 2021 semester.

\section{The Benjamin-Ono truncation}\label{s:BOtruncation}
In this section, we consider the frequency truncation of solutions to the Benjamin-Ono equation. Let $U$ be a solution to the Benjamin-Ono equation \eqref{BenjaminOnoEqn} whose initial data
satisfies 
\[
\|U_0\|_{H^m}\lesssim 1,
\]
for some positive integer $m$. We recall that the Benjamin-Ono equation is a completely integrable equation, and it has almost conserved $H^s$ norm for $s>\frac{1}{2}$; see for instance \cite{ifrim2019well}, \cite{kaup1998complete} and \cite{MR4270277}  . Not only is the $L^2$ norm of $U$ conserved,
but for each integer $m > 0$ there exists a conserved energy which 
at the leading order agrees with the $H^m$ norm of $U$. Hence the 
$H^m$ bound for the initial data transfers to the solution globally in time:
\[
\| U \|_{L_t^\infty H_x^m} \lesssim 1.
\]

For such a solution $U$, we construct a frequency localized approximate solution $\tilde U$ by 
\begin{equation}
    \tilde U = P_{\le b\epsilon^{-1}}U
\end{equation} 
for some small universal constant $b$.

Then $\tilde U$ satisfies the equation 
\begin{equation}
    \left(\partial_t - \frac{g^2}{c^3}H\partial_x^2\right)\tilde U=\lambda\tilde U \partial_x \tilde U +\tilde f,
\end{equation} 
for a forcing term $\tilde f$ given by 
\begin{equation*}
    \tilde f = P_{\le b\epsilon^{-1}}\left( \lambda U U_x\right)-\lambda P_{\le b\epsilon^{-1}} U \partial_x P_{\le b\epsilon^{-1}}U.
\end{equation*}

\begin{proposition} 
For the function $\tilde U$ defined above and $m\ge 2$, we have 
\begin{enumerate}[label=(\roman*)]
    \item The bounds \begin{equation}
        \|\tilde U\|_{\dot H^k}\lesssim \epsilon^{m-k}\label{Ucontrol}
    \end{equation} for $k\ge m$.
    
    \item The difference bounds \begin{equation}
        \|U-\tilde U\|_{L^2}\lesssim \epsilon^m\label{e:UcloseL2}
    \end{equation} and \begin{equation}
        \|U-\tilde U\|_{\dot H^1}\lesssim \epsilon^{m-1}.\label{e:UcloseH1}
    \end{equation}
    
    \item The error estimates \begin{equation}
        \|\tilde f\|_{L^2}\lesssim \epsilon^{m-1}\label{e:truncationerrorestimates}
    \end{equation}
    and \begin{equation}
        \|\tilde f\|_{\dot H^1}\lesssim \epsilon^{m-2}.\label{e:truncationerrorestimatesH1}
    \end{equation}
\end{enumerate}\label{p:frequencytruncation}
\end{proposition}

Note that since the projection to negative frequencies $\nP$ is bounded in all $L^p$ spaces, the estimates of this proposition hold for the negative-frequency Benjamin-Ono equation~\eqref{negativeFrequencyBO} as well as for the full Benjamin-Ono equation~\eqref{BenjaminOnoEqn}.

\begin{proof}
Without loss of generality, we assume here $b=1$.
The uniform $\dot{H}^k$ bounds come from the complete integrability of the Benjamin-Ono equation.
There are conserved energies for the Benjamin-Ono equation at each nonnegative integer regularity $k$, and we can apply Bernstein's inequality to control $\|\tilde U\|_{\dot H^k}$ in terms of $\|U\|_{\dot H^m}$ for $k\ge m$.

The uniform $L^2$ and $\dot H^1$ difference bounds come from the conserved $\dot H^m$ energy and an application of Bernstein's inequality.

It remains to prove the error estimates~\eqref{e:truncationerrorestimates}.
We write $U=U_{\ll \epsilon^{-1}}+U_{\gtrsim \epsilon^{-1}}$, and we view $\tilde f$ as the sum of the high-low, high-high, and low-high interactions, where low means frequencies $\ll\epsilon^{-1}$ and high means frequencies $\gtrsim \epsilon^{-1}$.
Note that the low-low interactions $\tilde f_{ll}$ vanish. 

\textbf{High-low interactions}: Discarding the outer projection and using Bernstein's inequality, 
we have
\begin{align*}
    \|\tilde f_{hl}\|_{L^2}&\approx\|U_{\gtrsim \epsilon^{-1}}\partial_x U_{\ll\epsilon^{-1}}\|_{L^2}\\
    &\le \|U_{\gtrsim \epsilon^{-1}}\|_{L^2}\|\partial_x U_{\ll \epsilon^{-1}}\|_{L^\infty}\\
    &\lesssim \epsilon^{-1}\|U_{\gtrsim \epsilon^{-1}}\|_{L^2}\|U_{\ll \epsilon^{-1}}\|_{L^\infty}.
\end{align*}
    We now use the uniform Sobolev control of $U$ and the low-frequency control of $U$ and conclude \begin{equation}
        \|\tilde f_{hl}\|_{L^2}\lesssim \epsilon^{m-1}.
    \end{equation}
    
    For the $\dot H^1$ estimate, either the derivative falls on $U_{\ll\epsilon^{-1}}$ and adds another power of $\epsilon^{-1}$ or it falls on $U_{\gtrsim\epsilon^{-1}}$ and we gain one fewer power of $\epsilon$ from the uniform Sobolev control of $U$; either way, we have \begin{equation}
                \|\tilde f_{hl}\|_{\dot H^1}\lesssim \epsilon^{m-2}.
    \end{equation}

\textbf{High-high interactions}: Discarding the outer projection and using the Gagliardo-Nirenberg interpolation inequality, 
we have
\begin{align*}
    \|\tilde f_{hh}\|_{L^2}&\approx\|U_{\gtrsim \epsilon^{-1}}\partial_x U_{\gtrsim\epsilon^{-1}}\|_{L^2}\\
    &\le \|U_{\gtrsim \epsilon^{-1}}\|_{L^2}\|\partial_x U_{\gtrsim\epsilon^{-1}}\|_{L^\infty}\\
    &\lesssim\|U_{\gtrsim \epsilon^{-1}}\|_{L^2}\|\partial_x U_{\gtrsim \epsilon^{-1}}\|_{L^2}^{1/2}\|\partial_x U_{\gtrsim \epsilon^{-1}}\|_{\dot H^1}^{1/2}.
\end{align*}

We now use the uniform Sobolev control of $U$ and conclude
\begin{equation}
    \|\tilde f_{hh}\|_{L^2}\lesssim \epsilon^{2m-\frac 3 2}.
\end{equation}

    For the $\dot H^1$ estimate, wherever the derivative falls, we gain one fewer power of $\epsilon$ from the uniform Sobolev control of $U$; we have \begin{equation}
                \|\tilde f_{hh}\|_{\dot H^1}\lesssim \epsilon^{2m-\frac 5 2}.
    \end{equation}

\textbf{Low-high interactions}: We keep the outer projection and apply a commutator bound: \begin{align*}
    \|\tilde f_{lh}\|_{L^2}&\approx\|P_{<\epsilon^{-1}}\left(U_{\ll \epsilon^{-1}}\partial_x U_{\gtrsim\epsilon^{-1}}\right)\|_{L^2}\\
    &\le \|\left[P_{<\epsilon^{-1}},U_{\ll \epsilon^{-1}}\right]\partial_x U_{\gtrsim \epsilon^{-1}}\|_{L^2}\\
    &\lesssim \|\partial_x U_{\ll \epsilon^{-1}}\|_{L^\infty}\|U_{\gtrsim\epsilon^{-1}}\|_{L^2}\\
    &\lesssim \epsilon^{-1}\|U_{\ll \epsilon^{-1}}\|_{L^\infty}\| U_{\gtrsim\epsilon^{-1}}\|_{L^2}.
\end{align*} We now use the uniform Sobolev control of $U$ and the low-frequency control of $U$ and conclude \begin{equation}
        \|\tilde f_{lh}\|_{L^2}\lesssim \epsilon^{m-1}.
    \end{equation}
    
    For the $\dot H^1$ estimate, either the derivative falls on $U_{\ll\epsilon^{-1}}$ and adds another power of $\epsilon^{-1}$ or it falls on $U_{\gtrsim\epsilon^{-1}}$ and means we gain one fewer power of $\epsilon$ from the uniform Sobolev control of $U$; either way, we have \begin{equation}
                \|\tilde f_{lh}\|_{\dot H^1}\lesssim \epsilon^{m-2}.
    \end{equation}

Putting these together, for $m\ge 2$, $m-1\le 2m-\frac 3 2$ and $m-2\le 2m-\frac 5 2$, therefore $\|\tilde f\|_{L^2}\lesssim \epsilon^{m-1}$ and $\|\tilde f\|_{\dot H^1}\lesssim \epsilon^{m-2}$.
\end{proof}

Corresponding to $\tilde{U}$ we define $\tilde{Y}^\epsilon$ via
\begin{equation}
    \tilde{Y}^\epsilon(t,x) = -i\nP[\epsilon \tilde{U}(\epsilon^2 t, \epsilon x - \frac{g}{c}\epsilon t)],\label{e:tildeyepsilondef}
\end{equation}
by \eqref{TransformOne}, \eqref{TransformTwo} and \eqref{TransformThree}. 
We collect the key properties of $\tilde Y^\epsilon$ which will be used in the sequel:

\begin{proposition}\label{p:tildeyepsilonprop}
 Let $U$ be a solution to the Benjamin-Ono equation~\eqref{BenjaminOnoEqn} whose initial data satisifies $\|U_0\|_{H^m}\lesssim 1$ for some positive integer $m \ge 2$, and let $\tilde U=P_{\le \epsilon^{-1}} U$.
 For the function $\tilde Y^\epsilon$ given by~\eqref{e:tildeyepsilondef}, the following hold: \begin{enumerate}[label=(\roman*)]
     \item Frequency localization: $\tilde Y^\epsilon$ is localized at frequencies in the range $(-2,0]$.
     
     \item $L^2$ and $L^\infty$ estimates showing frequency concentration at the $\epsilon$ scale: 
     \begin{equation}
        \begin{aligned}
            \|\partial^k \tilde Y^\epsilon\|_{L^2}\lesssim & \ \epsilon^{k+\frac 1 2},\quad 0\le k\le m,\\
            \|\partial^k \tilde Y^\epsilon\|_{L^2}\lesssim & \ \epsilon^{m+\frac 1 2},\quad m \le k,\\
            \|\partial^k\tilde Y^\epsilon\|_{L^\infty}\lesssim& \ \epsilon^{k+1},\quad 0\le k \le m-1,\\
            \|\partial^k\tilde Y^\epsilon\|_{L^\infty}\lesssim& \ \epsilon^{m+\frac 1 2},\quad m\le k.\\
        \end{aligned}\label{e:mainyestimates}
\end{equation}

    \item Closeness to the non-frequency-truncated rescaled solution $Y^\epsilon$: 
    \begin{equation}
        \|\tilde Y^\epsilon-Y^\epsilon\|_{H^1}\lesssim \epsilon^{m+\frac 1 2},\quad \|\tilde Y^\epsilon -Y^\epsilon\|_{L^\infty}\lesssim \epsilon^{m+\frac 1 2}.\label{e:L2truncationerror}
    \end{equation} 

    \item Velocity-shifted approximate Benjamin-Ono equation: 
    \begin{equation} \label{tY-eqn}
        \partial_t \tilde{Y}^\epsilon=-\frac g c \tilde{Y}^\epsilon_\alpha+\frac{ig^2}{c^3}\tilde{Y}^\epsilon_{\alpha\alpha}+ic\tilde{Y}^\epsilon\tilde{Y}^\epsilon_\alpha-ic\nP[\overline{\tilde{Y}^\epsilon}\tilde{Y}^\epsilon_\alpha]-ic\nP[\tilde{Y}^\epsilon\overline{\tilde{Y}^\epsilon_\alpha}]+\tilde f^\epsilon,
    \end{equation} 
    where $\tilde f^\epsilon$ is a small error term satisfying
    \begin{equation}
        \|\tilde f^\epsilon\|_{H^1}\lesssim \epsilon^{m+\frac 3 2},\quad \|\tilde f^\epsilon\|_{L^\infty}\lesssim \epsilon^{m+\frac 3 2}.\label{e:truncationerrorf}
    \end{equation}

 \end{enumerate}
\end{proposition}

\begin{proof}
Part (i) is immediate from the construction of $\tilde Y^\epsilon$.

The $L^2$ component of part (ii) follow from the construction of $\tilde Y^\epsilon$, the fact that $\|U\|_{H^m}\lesssim 1$, and the estimates~\eqref{Ucontrol} for $\tilde U$ for $k\ge m$.
The $L^\infty$ component then follows by interpolation.

The $H^1$ component of (iii) follows from~\eqref{e:UcloseL2} and \eqref{e:UcloseH1} after accounting for the rescalings in the definitions of $Y^\epsilon$ and $\tilde Y^\epsilon$; the $L^\infty$ part then follows by interpolation.

The main part of (iv) comes from the velocity shift and projection to negative frequencies being applied to the frequency-truncated Benjamin-Ono solution $\tilde U$.
For the error term $\tilde f$ in (iv), the $H^1$ estimate follows from~\eqref{e:truncationerrorestimates} and~\eqref{e:truncationerrorestimatesH1} after accounting for the rescalings in the definitions of $Y^\epsilon$ and $\tilde Y^\epsilon$; the $L^\infty$ estimate then follows by interpolation.
\end{proof}

Later we will use this proposition with $m\ge 3$.

From~\eqref{e:L2truncationerror} with $m\ge 3$, we have that the difference between the frequency-truncated approximate solution $\tilde Y^\epsilon$ and the exact solution $Y^\epsilon$ to~\eqref{negativeFrequencyBO} is small in all relevant norms, so it suffices to 
prove Theorem~\ref{t:maintheorem} with $Y^\epsilon$ replaced by its frequency truncation.

\section{Benjamin-Ono Approximation}\label{s:BOapproximate}

In this section, we prove Theorem \ref{t:section4theorem}, showing that  each Benjamin-Ono solution is naturally associated to an approximate water wave solution.

Our objective here is to show that, at the leading order, the dynamics of water waves which are close to the upper branch in the linear dispersion relation are governed 
by the projected Benjamin-Ono flow. 
The computations are at first heuristic, then formally justified later in the section. 

As the starting point for our analysis, we begin with a water wave solution
$(W,Q)$, for which we make the following assumptions:
\begin{itemize}
    \item $W$ and $Q$ are localized at frequency $\lesssim \epsilon$.
    
    \item Both $W$ and $Q$ have $L^2$ size $O(\epsilon^\frac12)$.
\end{itemize}
We remark that the second assumption is not justified by the  
energy norm, which would only give an $\dot H^\frac12$ bound for $Q$.
However, at the linear level at least this is consistent with the 
fact that we are looking for solutions along the upper branch of the 
dispersion relation, for which we expect to have $Q \sim \frac{g}{c} W$.

We also remark that, by Bernstein's inequality, the two assumptions above
imply that both $W$ and $Q$ have $L^\infty$ size $O(\epsilon)$. This will come
in handy when we estimate multilinear expressions below.

In this setting, given that we seek solutions on the cubic time scale 
$\epsilon^{-2}$, only source terms of $L^2$ size $\epsilon^{\frac52}$
and larger are significant. In effect, all terms we neglect will turn out to have $L^2$ size at most $\epsilon^{\frac72}$. For this reason, we introduce the notation
\[
 A \approx B  \quad \Longleftrightarrow \quad  A-B\in O_{L^2}(\epsilon^{\frac{7}{2}}).
\]
In particular, all the quartic and higher order terms in the original
system \eqref{vorticityEqn} can be discarded. 

Here we compute the cubic truncation of each of the nonlinear terms in the equation \eqref{vorticityEqn}: 
\begin{equation*}
\begin{aligned}
F^{(\leq 3)}  = & Q_\alpha - Q_\alpha W_\alpha - \nP[Q_\alpha\Bar{W}_\alpha - \Bar{Q}_\alpha W_\alpha] + Q_\alpha W_\alpha^2\\
& + \nP[Q_\alpha (|W_\alpha|^2+\Bar{W}_\alpha^2)] -\nP[\Bar{Q}_\alpha(W_\alpha^2 + |W_\alpha|^2)], \\
F_1^{(\leq 3)} = & W - \partial_\alpha \nP[|W|^2]+\nP[W\Bar{W}_\alpha^2 +\Bar{W}W_\alpha^2],\\
T_1^{(\leq 3)} = & \nP[WQ_\alpha - \Bar{W}Q_\alpha] + \nP[\Bar{W}W_\alpha Q_\alpha - W\Bar{W}_\alpha Q_\alpha].
\end{aligned}    
\end{equation*}
The superscript $(\leq 3)$ denotes terms up to cubic order in a formal power series expansion in $(W,Q)$.
Using this in \eqref{vorticityEqn}, we first observe that all cubic terms involve at least two derivatives, therefore they also have size $O_{L^2}(\epsilon^{\frac{7}{2}})$ and can be discarded.
This leaves us with only the quadratic nonlinearities, i.e. we
get
\begin{equation}
\left\{
             \begin{array}{lr}
            W_t +Q_\alpha \approx G^{(2)}  &\\
             Q_t +icQ - igW \approx K^{(2)},&  
             \end{array}
\right.\label{quasilinEqn}
\end{equation} 
where $G^{(2)}$ and $K^{(2)}$ represent the quadratic terms in $G$ and $K$, and are given by
\begin{equation*}
\begin{aligned}
G^{(2)} = & -i\frac{c}{2}\partial_\alpha \nP[|W|^2] +i\frac{c}{2}WW_\alpha + \nP[Q_\alpha \Bar{W}_\alpha - \Bar{Q}_\alpha W_\alpha] ,\\
K^{(2)} = & -Q_\alpha^2 - \nP[|Q_\alpha|^2] +i\frac{c}{2}WQ_\alpha + i\frac{c}{2}\nP[W\Bar{Q}_\alpha - \Bar{W}Q_\alpha].
\end{aligned}
\end{equation*}
One may simplify these expressions further by noting that the terms with two derivatives can also be discarded.

In  ~\cite{ifrim2018two}, the authors computed a normal form transformation which formally eliminates the 
quadratic terms in the equation. However, this can be used directly neither there nor here, as it is 
unbounded at low frequencies, which, incidentally, are exactly the frequencies we are interested in in this paper.
Instead,  here we give priority to the diagonalization of the problem, which is first carried 
out at the linear level. Precisely, we define two new variables which diagonalize the linear part of the 
system,
\begin{equation*}
Y^{+} = \frac{1}{2}\left(W -\frac{c - \sqrt{c^2+4g|D|}}{2g}Q\right), \quad Y^{-} = \frac{1}{2}\left(W -\frac{c + \sqrt{c^2+4g|D|}}{2g}Q\right),
\end{equation*}
so that
\begin{equation}
\left\{
             \begin{array}{lr}
          (i\partial_t- \frac{c - \sqrt{c^2 + 4g|D|}}{2} )Y^{+} \approx \frac i 2 G^{(2)}(W,Q)- i\left(\frac{c-\sqrt{c^2+4g|D|}}{4g}\right)K^{(2)}(W,Q)  \\
             (i\partial_t- \frac{c +\sqrt{c^2 + 4g|D|}}{2} )Y^{-} \approx \frac i 2 G^{(2)}(W,Q)-i\left(\frac{c+\sqrt{c^2+4g|D|}}{4g}\right)K^{(2)}(W,Q).
             \end{array}
\right.\label{e:systemforYs}
\end{equation}
Here $Y^+$ corresponds to the waves which move to the right, i.e. the upper branch of the dispersion relation in Figure~\ref{dr}, and is the component which we seek to approximate with Benjamin-Ono waves.
$Y^-$, on the other hand, corresponds to the waves which move to the left, i.e. the lower branch of the dispersion relation in Figure~\ref{dr}, and is the component which we want to be of smaller size,
namely $O_{L^2}(\epsilon^{\frac52})$.

Writing 
\begin{equation}
    W-\frac{c}{2g} Q=Y^++Y^-,\quad \frac{\sqrt{c^2+4g|D|}}{2g} Q=Y^+-Y^-,
\end{equation} 
and using the first-order approximations for the square root, which hold to $O_{L^2}(\epsilon^2)$ and are thus accurate enough when working with quadratic and higher-order terms,
\begin{equation*}
    \sqrt{c^2+4g|D|}\approx c+\frac{2ig}{c}\partial_\alpha,\qquad \sqrt{c^2+4g|D|}^{-1}\approx \frac 1 c-\frac{2ig}{c^3}\partial_\alpha,
\end{equation*} we have, \begin{align}
    W=& Y^++Y^-+\left(1-\frac{2ig}{c^2}\partial_\alpha\right)\left(Y^+-Y^-\right)=2Y^+-\frac{2ig}{c^2}Y^+_\alpha+\frac{2ig}{c^2}Y^-_\alpha+O_{L^2}(\epsilon^{\frac 5 2}),\label{e:WinYterms}\\
    Q=& \left(\frac{2g}{c}-\frac{4ig^2}{c^3}\partial_\alpha\right)\left(Y^+-Y^-\right)+O_{L^2}(\epsilon^{\frac 5 2}).\label{e:QinYterms}
\end{align}

We expand and express the right-hand side of~\eqref{e:systemforYs} in terms of $Y^+$ and $Y^-$, treating $Y^-$ provisionally as of size $O_{L^2}(\epsilon^{\frac 5 2})$ and ignoring all terms of order $O_{L^2}(\epsilon^{\frac 7 2})$ or smaller. This yields
\begin{align*}
          \left(i\partial_t- \frac{c - \sqrt{c^2 + 4g|D|}}{2} \right)Y^{+} &\approx  c\nP[Y^+\overline{Y^+_\alpha}+\overline{Y^+}Y^+_\alpha-Y^+Y^+_\alpha]\\
             \left(i\partial_t- \frac{c +\sqrt{c^2 + 4g|D|}}{2}\right)Y^{-} & \approx 2c\nP\left[Y^+\overline{Y^+_\alpha}\right].
\end{align*}

In the first equation we can harmlessly replace the dispersion relation with its quadratic approximation,
which leads to the same equation as the velocity shifted Benjamin-Ono flow in \eqref{negativeFrequencyBO}, where we set $\lambda=c$. Then it is natural to choose $Y^+ = Y^\epsilon$ to be a solution of 
\begin{equation*}
     (i\partial_t+\omega_0(D))Y^\epsilon\approx c \left(- Y^\epsilon Y^\epsilon_\alpha+\nP\left[Y^\epsilon \overline{Y}^\epsilon_\alpha\right]+\nP\left[\overline{Y}^\epsilon Y^\epsilon_\alpha\right]\right).
\end{equation*}

Making the further ansatz that $Y^-$ is a polynomial (possibly also involving projections) in $Y^+$ and $Y^+_\alpha$ of degree at least two, the first equation shows that $\partial_t Y^-$ must be of size $O(\epsilon^{\frac 7 2})$.
This makes the second equation above into \begin{equation*}
    -cY^-\approx 2c\nP\left[Y^+\overline{Y^+_\alpha}\right],
\end{equation*} motivating us to set \begin{equation}
    Y^-=-2\nP\left[Y^+\overline{Y^+_\alpha}\right].\label{Y-definition}
\end{equation}
This validates our claim that $Y^-$ should be of lower order compared to $Y^+$, precisely $O_{L^2}(\epsilon^{\frac52})$.
We now proceed with the rigorous computation.

\subsection{Constructing the approximate water waves solution in Theorem~\ref{t:ww-approx}}
Let $Y^\epsilon$ be the solution to~\eqref{negativeFrequencyBO} produced in Section~\ref{s:BOtruncation}, and let $\tilde{Y^\epsilon}=P_{\le 1}Y^\epsilon$. This is a frequency-localized function which is an approximate solution to the negative-frequency Benjamin-Ono equation~\eqref{negativeFrequencyBO}.
Good bounds for this approximate solution are proved in Section~\ref{s:BOtruncation}.

Inspired by~\eqref{e:WinYterms}-\eqref{Y-definition} but working with quadratic approximations to the $1/2$th order differential operator to get $O_{H^1}(\epsilon^2)$ accuracy in the linear part of the equations, we define our candidate $(W^\epsilon,Q^\epsilon)$ for the approximate solution by
\begin{equation}
\begin{aligned}
    W^\epsilon&=2\tilde{Y}^\epsilon-\frac{2ig}{c^2}\tilde{Y}^\epsilon_\alpha-\frac{6g^2}{c^4}\tilde{Y}^\epsilon_{\alpha\alpha},\\
    Q^\epsilon&=\frac{2g}{c}\tilde{Y}^\epsilon-\frac{4ig^2}{c^3}\tilde{Y}^\epsilon_\alpha-\frac{12g^3}{c^5}\tilde{Y}^\epsilon_{\alpha\alpha}+\frac{4g}{c}\nP\left[\tilde Y^\epsilon\overline{\tilde Y^\epsilon_\alpha}\right].
    \end{aligned} \label{approxWWsoln}
\end{equation}

We claim that $\left(W^\epsilon,Q^\epsilon\right)$ is a good approximate solution to the constant vorticity water wave system.

\begin{proof}[Proof of Theorem~\ref{t:ww-approx}]
Part (i) follows from the construction of $W^\epsilon$ and $Q^\epsilon$ and the estimates~\eqref{e:mainyestimates} for $\tilde Y_\epsilon$.

\medskip

Part (ii) follows directly from the construction of $\tilde Y^\epsilon$, which is frequency localized 
in $[-1,0]$, and  satisfies the estimates~\eqref{e:mainyestimates}. 
Since $W^\epsilon$ and $Q^\epsilon$ consist of sums of products of up to two factors of $\tilde Y^\epsilon$ and its conjugate (with projections as needed to ensure negative frequencies for the output), it follows they are strictly localized to frequencies in the range $(-2,0]$, and satisfy similar estimates to $\tilde Y^\epsilon$.
In particular, we have \begin{equation*}
    \|\W^\epsilon\|_{L^\infty}\lesssim \epsilon,\quad \|\W^\epsilon\|_{\dot H^j}\lesssim\epsilon^{\frac 3 2+n},\quad \|Q_\alpha^\epsilon\|_{\dot H^{j+\frac 1 2}}\lesssim \epsilon^{2+n},\quad 0\le j\le m-2
\end{equation*} and at the top order \begin{equation*}
    \|W^\epsilon\|_{\dot H^{m-1}}\lesssim \epsilon^{m+\frac 1 2},\quad \|Q_\alpha^\epsilon\|_{\dot H^{m-\frac 1 2}}\lesssim \epsilon^{m+\frac 1 2}
\end{equation*}    
so the estimates~\eqref{e:firststepfreqconcentration} hold.

\medskip

For part (iii), our starting point is the approximate velocity shifted Benjamin-Ono 
equation \eqref{tY-eqn} for $\tilde Y_\epsilon$, which we recall here:
\begin{equation}\label{tY-eqn-re}
    \partial_t \tilde{Y}^\epsilon=-\frac g c \tilde{Y}^\epsilon_\alpha+\frac{ig^2}{c^3}\tilde{Y}^\epsilon_{\alpha\alpha}+ic\tilde{Y}^\epsilon\tilde{Y}^\epsilon_\alpha-ic\nP[\overline{\tilde{Y}^\epsilon}\tilde{Y}^\epsilon_\alpha]-ic\nP[\tilde{Y}^\epsilon\overline{\tilde{Y}^\epsilon_\alpha}]+\tilde f^\epsilon,
\end{equation} 
Here, the estimates~\eqref{e:truncationerrorf} show that for $m\ge 3$, 
the source term $\tilde f^\epsilon$ satisfies
\begin{equation*}
    \tilde f^\epsilon = O_{L^2}(\epsilon^{\frac 7 2}),\quad \tilde f^\epsilon = O_{\dot H^1}(\epsilon^{\frac 7 2}),\quad  \tilde f^\epsilon = O_{L^\infty}(\epsilon^{\frac 7 2}),
\end{equation*} 
so its contributions can safely be absorbed into the $O(\epsilon^{\frac 7 2})$ error in both $L^2$ and $\dot H^1$.

Then we compute, using \eqref{tY-eqn-re}, the estimates~\eqref{e:mainyestimates} to absorb the lower-order terms into the error, 
\begin{align*}
    \partial_t W^\epsilon&=2\partial_t\tilde{Y}^\epsilon-\frac{2ig}{c^2}\partial_t\tilde{Y}^\epsilon_\alpha-\frac{6g^2}{c^4}\partial_t\tilde{Y}^\epsilon_{\alpha\alpha}\\
    &=-\frac{2g}{c}\tilde{Y}^\epsilon_\alpha+\frac{2ig^2}{c^3}\tilde{Y}^\epsilon_{\alpha\alpha}+2ic\tilde{Y}^\epsilon\tilde{Y}^\epsilon_\alpha-2ic\nP[\overline{\tilde{Y}^\epsilon}\tilde{Y}^\epsilon_\alpha]-2ic\nP[\tilde{Y}^\epsilon\overline{\tilde{Y}^\epsilon_\alpha}]+\frac{2ig^2}{c^3}\tilde{Y}^\epsilon_{\alpha\alpha}+O_{H^1}(\epsilon^{\frac 7 2}).
    \end{align*}
Here all linear terms with at least three derivatives, all quadratic terms with at least two derivatives
and all cubic terms with at least one derivative qualify as admissible errors.
    
    Similarly, we have 
    \begin{equation*}
        Q^\epsilon_\alpha=\frac{2g}{c}\tilde{Y}^\epsilon_\alpha-\frac{4ig^2}{c^3}\tilde{Y}^\epsilon_{\alpha\alpha}+O_{H^1}(\epsilon^{\frac 7 2}),
    \end{equation*}
  so we obtain
  \begin{equation*}
        W^\epsilon_t+Q^\epsilon_\alpha=-2ic\nP[\tilde Y^\epsilon\overline{\tilde Y^\epsilon_\alpha}+\overline{\tilde Y^\epsilon}\tilde Y^\epsilon_\alpha-\tilde Y^\epsilon\tilde Y^\epsilon_\alpha]+O_{H^1}(\epsilon^{\frac 7 2}).
    \end{equation*}
    
    We expand the nonlinearity $G^{(2)}$ given above, writing $W^\epsilon$ and $Q^\epsilon$ in terms of $\tilde Y^\epsilon$,
    \begin{align*}
        G^{(2)}(W^\epsilon,Q^\epsilon)&=2ic\nP[\tilde Y^\epsilon\tilde Y^\epsilon_\alpha-\tilde Y^\epsilon\overline{\tilde Y^\epsilon_\alpha}-\overline{\tilde Y^\epsilon}\tilde Y^\epsilon_\alpha].
    \end{align*} 
    Using~\eqref{e:mainyestimates} to control the size of the higher order terms in $G$, this means that \begin{equation*}
        W_t^\epsilon+Q_\alpha^\epsilon = G(W^\epsilon, Q^\epsilon)+ g^\epsilon,
    \end{equation*} 
    with $\|g^\epsilon\|_{H^1}\lesssim \epsilon^{\frac 7 2}$ as desired.

Consider similarly $Q^\epsilon_t +icQ^\epsilon-igW^\epsilon$, where we again discard all terms that will, using the estimates~\eqref{e:mainyestimates} as before, be of size $O(\epsilon^{\frac 7 2})$ in $H^1$ and thus $O(\epsilon^{\frac 7 2})$ in $\dot H^{\frac 1 2}$, our construction gives
\begin{equation*}
        Q^\epsilon_t+icQ^\epsilon-igW^\epsilon=2ig\nP[\tilde Y^\epsilon\overline{\tilde Y^\epsilon_\alpha}-\overline{\tilde Y^\epsilon}\tilde Y^\epsilon_\alpha+\tilde Y^\epsilon \tilde Y^\epsilon_\alpha]+O_{H^1}(\epsilon^{\frac 7 2}).
\end{equation*} 
    
Expanding the nonlinearity $K^{(2)}$ above, writing $W^\epsilon$ and $Q^\epsilon$ in terms of $Y^\epsilon$, we have \begin{align*}
    K^{(2)}(W^\epsilon, Q^\epsilon)=2ig\nP[\tilde Y^\epsilon \tilde Y^\epsilon_\alpha-\overline{\tilde Y^\epsilon}\tilde Y^\epsilon_\alpha+\tilde Y^\epsilon\overline{\tilde Y^\epsilon_\alpha}].
\end{align*}
Using~\eqref{e:mainyestimates} to control the size of the differentiated cubic and higher order terms, this means that \begin{equation*}
    Q^\epsilon_t+icQ^\epsilon-ig W^\epsilon=K(W^\epsilon,Q^\epsilon)+ k^\epsilon,
\end{equation*} with $\|k^\epsilon\|_{\dot H^{\frac 1 2}}\lesssim \epsilon^{\frac 7 2}$ as desired.

\end{proof}

\section{Refined Energy Estimates}\label{s:refinedenergy}
In this section, we prove a more refined version of Proposition~\ref{t:oldcubicenergy} from \cite{ifrim2018two}.
This will allow us to better propagate the $\epsilon$ frequency concentration of our solutions, which in turn will
be helpful for controlling the high-frequency part of the difference between the approximate and exact water wave solutions in the next section.

\begin{proposition}\label{p:improvedenergyestimate}
 The cubic energy $E^{n,(3)}(\W,R)$ satisfies the following estimate: 
 \begin{align}
     \frac{d}{dt}E^{n,(3)}\lesssim_{\uA}& \left(c^2AB+c^3A^2+c^3BA_{-1/2}+c^4AA_{-1/2}+c^4A_{-1}B+c^5 AA_{-1}\right)(E^{n,(3)}E^{n-1,(3)})^{\frac 1 2}\nonumber\\
     &+\uA\uB E^{n,(3)},\label{strongcubicest}
 \end{align} 
 with $E^{-1,(3)}$ replaced with $\mathcal E(W,Q)$.
\end{proposition}

\begin{proof}
We recall the order calculus for multilinear expression in~\cite{ifrim2018two}, where the \textit{order} of single expressions is defined according to the following scheme: 
\begin{itemize}
    \item $W^{(k)}$ has order $k-1$
    
    \item $R^{(k)}$ has order $k-\frac 1 2$
    
    \item $c$ has order $\frac 1 2$
\end{itemize}
The \textit{order} of an integral form is defined as the sum of the orders of its factors, 
and the \textit{leading order} is the largest sum of the orders of two terms in the integrand with opposite conjugations. 
While working in the low-frequency regime removes the heuristic relationship between this order calculus and the scaling of the terms, it remains a useful bookkeeping device here.
The energy $E^{n,(3)}$ is of order $2n$, and $\frac{d}{dt}E^{n,(3)}$ is of order $2n+\frac 1 2$.

We use the cubic energies $E^{n,(3)}$ introduced in Section 4.4 of~\cite{ifrim2018two}, and we analyze the computations done there for the time derivative of this cubic energy more closely.
The most difficult part of the cubic energy $E^{n,(3)}$ is built starting from the following five components:
\begin{align*}
    I_0=&2\Re\int -\overline \W^{(n)}\partial^{n+1}\left[(W+\overline W)\W\right]+i\overline Q^{(n+2)}\partial^{n+1}\left[(W+\overline W)Q_\alpha\right]d\alpha,\\
    I_1=&c\Re\int -\overline \W^{(n)}\partial^{n+1}(\W(Q+\overline Q)+Q_\alpha(W+\overline W))+\overline Q^{(n+2)}\partial^{n+1}(\frac 1 2 W^2+|W|^2)\\
    &+\frac i g\overline Q^{(n+2)}\partial^{n+1}(QQ_\alpha+\overline QQ_\alpha)d\alpha, \\
    I_2=&c^2\Re \int i\overline{\W}^{(n)}\partial^{n+1}\left(\W\partial^{-1}(W-\overline W)+W^2+\frac 1 2 |W|^2\right)-\frac{1}{2g}\overline{\W}^{(n)}\partial^{n+1}(QQ_\alpha+\overline QQ_\alpha)\\
    &+\frac 1 g \overline{Q}^{(n+2)}\partial^{n+1}(Q_\alpha (\partial^{-1}W-\partial^{-1}\overline W))+\frac{1}{2g}\overline Q^{(n+2)}\partial^{n+1}(WQ+W\overline Q)d\alpha, \\
    I_3=&\frac{c^3}{2g}\Im \int \overline{Q}^{(n+2)} \partial^{n+1}\left[\partial^{-1}(W-\overline W)W\right] -\overline{\W}^{(n)}\partial^{n+1}\left[W(Q+\overline Q)+\partial^{-1}(W-\overline W)Q_\alpha\right]d\alpha, \\
    I_4=&\frac{c^4}{2g^2}\Re\int \W^{(n)}\partial^{n+1}\left[(\partial^{-1}W-\partial^{-1}\overline W)W\right]d\alpha.
\end{align*} 
As explained in the proof of Lemma 4.7 in~\cite{ifrim2018two}, these terms contain no nonzero terms involving $\partial^{-1}W$, $\partial^{-1}\overline W$, or undifferentiated $Q$, so we can replace all the $Q_\alpha$ terms with $R$ in the final definition of the energy.
The terms in $I_0$ through $I_4$ where there is no $\partial^{-1}W$, $\partial^{-1}\overline W$, or undifferentiated $Q$ but all the derivatives from $\partial^{n+1}$ fall on the same term, for example the \begin{equation*}
    \frac{c^3}{2g}\int i(2n+4)(W-\overline W)\overline\W^{(n)}Q^{(n+1)}\, d\alpha
\end{equation*} component of $I_3$, are considered high-frequency terms and are handled separately with a quasilinear modification, and their time derivatives are bounded by $\uA\uB E^{n,(3)}$.
Additionally, all other components of $E^{n,(3)}$ have time derivatives bounded by $\uA\uB E^{n,(3)}$.

By using the fact that \begin{equation}
    \frac{d}{dt}\int f_1f_2f_3 d\alpha=\int (\partial_t +\underline b\partial_\alpha)f_1f_2f_3+f_1(\partial_t+\underline b\partial_\alpha)f_2f_3+f_1f_2(\partial_t+\underline b\partial_\alpha)f_3-{\underline b}_\alpha f_1f_2f_3\, d\alpha\label{cubictoquartic}
\end{equation} and the construction of the energies, we see that the cubic terms in the time derivative of the energy are zero.

Our analysis mirrors that in the proof of Lemma 4.8 in~\cite{ifrim2018two}, but here we get a clearer picture of how terms with a factor of $(E^{n-1,(3)})^{1/2}$ can look.

The following bounds, obtained by using H\"older's inequality and Gagliardo-Nirenberg, will be central to our analysis.
For all $k\ge 0$, $j\ge k$, we have \begin{align}
    \|W^{(k)}W^{(j-k)}\|_{L^p}\lesssim& A_{-1}\|W^{(j)}\|_{L^p},\label{e:interptwoW}\\
    \|\W^{(k)} R^{(j-k)}\|_{L^p}\lesssim& A_{-1/2}\|R\|_{\dot W^{j+\frac 1 2,p}},\\
    \|R^{(k)}R^{(j-k)}\|_{L^p}\lesssim& A_{-1/2}\|R^{j}\|_{L^p}
\end{align}
and, for all $k\ge 1$, $j>k$, \begin{equation}
    \|R^{(k)}R^{(j-k)}\|_{L^p}\lesssim A\|R\|_{\dot W^{j-\frac 1 2,p}}.\label{e:interptwoR}
\end{equation}

Here and in much of the sequel, we ignore projections and conjugates (which are bounded on $L^P$ spaces), except when they are helpful in moving derivatives (e.g. viewing $\nP[W\overline R_\alpha]$ as $\W R$ for the purposes of estimates).

When we take a transport derivative $\partial_t +\underline b\partial_\alpha$ of $W^{(j)}$, the structure of the (differentiated if necessary) constant vorticity water wave equations tells us that the can take one of a few forms, up to the multiplication by factors of $\frac{1}{1+\W}$ or its conjugate, which we view as $(1-Y)$ or $(1-\overline Y)$ and thus bounded by a constant:
\begin{itemize}
    \item Multiplication of the $W^{(j)}$ factor by $c\W$,
    
    \item Multiplication of the $W^{(j)}$ factor by $R_\alpha$,
    
    \item Replacement of the $W^{(j)}$ factor by $\W_\alpha R^{(j)}$,
    
    \item Replacement of the $W^{(j)}$ factor by $\W R^{(j+1)}$, or
    
    \item Replacement of the $W^{(j)}$ factor by a product of lower-order terms, denoted $\mathbf{err}(L^2)$, which can be bounded in $L^p$ by either 
    \[
    B\|(\W,R)\|_{\dot W^{n,p}\times \dot W^{n+\frac 1 2,p}}
    \]
    or 
    \[
    cA\|(\W,R)\|_{\dot W^{n,p}\times \dot W^{n+\frac 1 2,p}}.
    \]
    
\end{itemize}

Similarly, when we take a transport derivative of $R^{(j)}$, one of the following things can happen: \begin{itemize}
    \item Multiplication of the $R^{(j)}$ factor by $c\W$,
    
    \item Multiplication of the $R^{(j)}$ factor by $R_\alpha$,
    
    \item Replacement of the $R^{(j)}$ factor by $cR\W^{(j)}$, or
    
    \item Replacement of the $R^{(j)}$ factor by a product of lower-order terms, denoted $\mathbf{err}(\dot H^{\frac 1 2})$, which can be bounded in $L^p$ by either 
    \[
    A\|(\W,R)\|_{\dot W^{n,p}\times \dot W^{n+\frac 1 2,p}}
    \] or 
    \[
    cA_{-1/2}\|(\W,R)\|_{\dot W^{n,p}\times \dot W^{n+\frac 1 2,p}}.
    \]
\end{itemize}

The factor $\underline b_\alpha$ can contribute one of the following (up to factors of $\W$, $Y$, $(1-Y)$, and their conjugates, which are bounded by $A$ or a constant): \begin{itemize}
    \item A factor of $c\W$ or
    
    \item A factor of $R_\alpha$.
\end{itemize}

Looking at the transport derivative of (or $\underline b_\alpha$ times) the low-frequency components of $E^{n,(3)}$ above, we see that all components either contain a $\W^{(n)}$ factor, which can be directly bounded in $L^2$ by $(E^{n,(3)})^{\frac 1 2}$, or, possibly after an integration by parts, an $R^{(n)}$ factor.

By applying H\"older's inequality and interpolation on $\W R^{(n)}$ or $R_\alpha R^{(n)}$,  we can bound the $R^{(n)}$ factor and a piece of the coefficient we gain in $L^2$ by $A_{-\frac 1 2}(E^{n,(3)})^{\frac 1 2}$ or $A(E^{n,(3)})^{\frac 1 2}$.
In any case, we are guaranteed to have a factor of $(E^{n,(3)})^{\frac 1 2}$ in any component of $\frac{d}{dt} E^{n,(3)}$.

The remaining factors aside from $(E^{n,(3)})^{\frac 1 2}$ are of total order $n+\frac 1 2$.
The combined order of the two lowest-order factors from the low-frequency part of $E^{n,(3)}$ is between $n-2$ and $n$.
We can use H\"older's inequality and interpolation as in~\eqref{e:interptwoW}-~\eqref{e:interptwoR} to control their contribution (possibly after one of them has had a transport derivative applied) by either $(E^{n,(3)})^{\frac 1 2}$ or $(E^{n-1,(3)})^{\frac 1 2}$ with some coefficient.

So any term in $\frac{d}{dt}E^{n,(3)}$ can either be controlled by $\uA\uB E^{n,(3)}$ by the arguments in~\cite{ifrim2018two}, or it can be controlled by an expression of either the form \begin{equation*}
    c^jA_{\beta}A_\gamma E^{n,(3)}
\end{equation*} with $\frac j 2+\beta+\gamma=\frac 1 2$ or the form \begin{equation*}
    c^jA_{\beta}A_\gamma (E^{n,(3)}E^{n-1,(3)})^{\frac 1 2},
\end{equation*} with $\frac j 2 + \beta + \gamma=\frac 3 2$, where $\beta$ and $\gamma$ are $-1$, $-\frac 1 2$, $0$, or $\frac 1 2$ and we write $B$ as $A_{\frac 1 2}$.

Since terms in $E^{n,(3)}$ have at most a factor of $c^4$ and taking a transport derivative or multiplying by $\underline b_\alpha$ can add at most one factor of $c$, we know $0\le j\le 5$.
Moreover, since there is either one $W^{j}W^{n-j}$ or $W^{j}W^{n-j+1}$ pair or at most one factor of $W$ in the original terms from the energy, we cannot have both $\beta$ and $\gamma$ equal to $-1$.
These order considerations mean that in the first case, we have only terms like $c^4A_{-1}A_{-\frac 1 2}$ which are bounded by $\uA\uB$ and in the second case, we have exactly the other terms in~\eqref{strongcubicest}. 

As an example, we treat below a single type of term coming from $I_3$.

After harmlessly replacing $Q_\alpha$ with $R$, one type of term from $I_3$ is \begin{equation}
    \qquad I_{3,2}^j=\frac{c^3}{2g}\int R^{(n+1)}W^{(j)}W^{(n-j)}d\alpha
\end{equation} with various signs and placements of conjugations which will not be important for this analysis.

Looking in detail only at the terms coming from when this term is multiplied by $\underline{b}_\alpha$ when taking the time derivative of the energy, we have
\begin{equation}
    \underline b_\alpha=\nP\left[\frac{R_\alpha}{1+\overline \W}\right]+\overline{\nP}\left[\frac{\overline R_\alpha}{1+\overline \W}\right]-\frac{ic}{2}\nP\left[\frac{\W}{1+\overline{\W}}\right]+\frac{ic}{2}\overline{\nP}\left[\frac{\overline{\W}}{1+\W}\right].\label{balpha}
\end{equation}

Since we are only looking to improve the $c^4$ terms in $\frac{d}{dt}E^{n,(3)}$, we only need the components of $\underline b_\alpha$ that involve $c$, which are of the form $\frac{ic}{2}\nP[\W(1-\overline Y)]$ or its conjugate.
We first integrate by parts one derivative off of $R^{(n+1)}$, then three different types of terms can arise:
\begin{itemize}
    \item For those terms of the form $\W(1-\overline Y)R^{(n)}W^{(j)}W^{(n-j+1)}$, we can bound $1-Y$ in $L^\infty$ by $2$, then apply H\"older's inequality and interpolation to bound $\W R^{(n)}$ in $L^2$ by $A_{-1/2}(E^{n,(3)})^{1/2}$ and $\partial^j W\partial^{n-j+1}W$ in $L^2$ by $A_{-1}(E^{n,(3)})^{1/2}$.
    Since we have $4$ powers of $c$ and $\uA$ contains $c^2A_{-1}$ and $\uB$ contains $c^2A_{-1/2}$, we don't obtain an improvement for these terms.
    These terms can thus be bounded by the $\uA\uB E^{n,(3)}$ term in~\eqref{strongcubicest}.
    
    \item For those terms of the form $\W_\alpha (1-\overline Y)R^{(n)}W^{(j)}W^{(n-j)}$, we can bound $1-Y$ in $L^\infty$ by $2$, apply H\"older's inequality and interpolation to bound $\W_\alpha R^{(n)}$ in $L^2$ by $B(E^{n,(3)})^{1/2}$ and $\partial^j W\partial^{n-j}W$ in $L^2$ by $A_{-1}(E^{n-1,(3)})^{1/2}$.
    These terms can thus be bounded by the $c^4A_{-1}B(E^{n,(3)}E^{n-1,(3)})^{1/2}$ term in~\eqref{strongcubicest}.
    
    \item For those terms of the form $\W (1-\overline Y)_\alpha R^{(n)}W^{(j)}W^{(n-j)}$, we have $Y_\alpha=\W_\alpha (1-Y)^2$, so that we can estimate the $(1-Y)$ factors by $2$ and the $\W$ factor in $L^\infty$ by $A$, then apply H\"older's inequality and interpolation to bound $\W_\alpha R^{(n)}$ in $L^2$ by $B(E^{n,(3)})^{1/2}$ and $\partial^j W\partial^{n-j}W$ in $L^2$ by $A_{-1}(E^{n-1,(3)})^{1/2}$.
    These terms can thus be bounded by the $c^4A_{-1}B(E^{n,(3)}E^{n-1,(3)})^{1/2}$ term in~\eqref{strongcubicest}.
\end{itemize}

\end{proof}

\section{A bootstrap argument} \label{s: bootstrap}

In this section, we  compare the approximate solution $(W^\epsilon, Q^\epsilon)$ with the exact solution $(W, Q)$ having the same initial data.

Consider the two-parameter family of functions $\left(W(t,s),Q(t,s)\right)$ for $(t,s)\in [0,T/\epsilon^2]\times[0,T/\epsilon^2]$ given by solving the 2D water wave equations with constant vorticity with initial condition at time $t=s$, 
\begin{equation}
    W(s,s)=W^\epsilon(s),\quad Q(s,s)=Q^\epsilon(s).
\end{equation} 
We do not know \textit{a priori} that we can solve the constant vorticity water wave equations starting from $t=s$ on the whole time interval $[0,T/\epsilon^2]$, but we know, by Theorem 1 from~\cite{ifrim2018two}, that as long as the control norms $\uA$ and $\uB$ stay bounded, we can continue the solution.
We will use a bootstrap argument to conclude that the solution exists and satisfies good bounds on the whole square $(t,s)\in [0,T/\epsilon^2]\times[0,T/\epsilon^2]$.

By energy estimates for the negative frequency Benjamin-Ono equation and Proposition~\ref{p:tildeyepsilonprop}, we know that 
\begin{equation*}
    \|\left(W(s,s),Q(s,s)\right)\|_{\mathcal{H}}\lesssim \epsilon^{\frac 1 2}.
\end{equation*}
We also have similar bounds for higher Sobolev norms, frequency localization to frequencies $\lesssim \epsilon$, and pointwise bounds.
These imply that along the diagonal of $[0,T/\epsilon^2]\times[0,T/\epsilon^2]$, 
\begin{equation}
    A(s,s)\lesssim\epsilon^2,\quad B(s,s)\lesssim\epsilon^{\frac 5 2},\quad A_{-\frac 1 2}\lesssim \epsilon^{\frac 3 2},\quad A_{-1}\lesssim\epsilon,\quad \|R\|_{L^\infty}\lesssim \epsilon^2.
\end{equation}
Note that the bounds for $R$ are as strong as those for $\W$ because in our setting, $R$ and $\W$ have the same scaling, while in general $R$ is half a derivative higher than $\W$. See also the comment after 
\eqref{55}.

We now set up a bootstrap argument.
By the continuity of the control norms and pointwise norms, there is a neighborhood of the diagonal where, for some large $M$ independent of $T$, \begin{equation}
        \|W\|_{L^\infty}\le M\epsilon, \||D|^{\frac 1 2}W\|_{L^\infty}\le M\epsilon^{\frac 3 2},\|\W\|_{L^\infty}\le M\epsilon^2,\|R\|_{L^\infty}\le M\epsilon^2, A\le M\epsilon^2, B\le M\epsilon^{\frac 5 2}.\label{bootstrapconclusion}
\end{equation}
We claim that as long as \begin{equation}
        \|W(t,s)\|_{L^\infty}\le 2M\epsilon,\quad \uB(t,s)\le 2M\epsilon,\quad\|\W\|_{L^\infty}\le 2M\epsilon^2,\quad\|R\|_{L^\infty}\le 2M\epsilon^2,\label{bootstrapweaker}
\end{equation} the stronger estimates~\eqref{bootstrapconclusion} hold.
By continuity, this gives $\uA\le M\epsilon$ and $\uB\le M\epsilon$ on all of $[0,T/\epsilon^2]\times[0,T/\epsilon^2]$, and thus guarantees the existence of solutions $(W(\cdot,s),Q(\cdot,s))$ on all of $[0,T/\epsilon^2]$ for all $s$. 

For the rest of this section, implicit constants can depend on $c$, $g$, and the control norms, but not on $\epsilon$ or the large constant $M$.

We will split the bootstrap bound into two parts: the energy estimates for the linearized equation to control the ``low-frequency'' component and the energy estimates for the full problem to control the ``high-frequency'' component.

To study the behavior of the family $(W(t,s),Q(t,s))$ in the $s$ direction, we need to estimate
the functions $(w(t,s),q(t,s)) = (\partial_s W(t,s),\partial_s Q(t,s))$, 
which have initial data at time $t=s$ given by
\[
(w(s,s),q(s,s)) = (g^\epsilon(s),k^\epsilon(s)).
\]
Here we immediately switch to the good variables 
\[
(w(t,s), r(t,s)), \qquad r(t,s)= q(t,s)- R(t,s) w(t,s).
\]
As functions of $t$, these solve the linearized equations~\eqref{linearizedeqn} of the constant vorticity water wave system  around $(W(\cdot,s),Q(\cdot,s))$. 

According to Proposition~\ref{t:linearizedestimates}, we have a linearized energy functional $E_{lin}^{(3)}$, equivalent to $\|(w,r)\|_{\mathcal H}^2$ such that, for large enough $M$ and some constant $C$, \begin{equation}\label{55}
  \frac{d}{dt}E_{lin}^{(3)}(w,r)\le 4CM^2\epsilon^2 E_{lin}^{(3)}(w,r).
\end{equation}
Note that the stronger bootstrap assumption on $\| R\|_{L^\infty}$ is essential here in order to get a cubic estimate for the linearized energy, since unlike the irrotational case in \cite{ifrim2019nls}, we do not always have cubic linearized energy estimates in the constant vorticity case.

By Gronwall's inequality, applying such an estimate over a time interval of size at most $T\epsilon^{-2}$ gives \begin{equation}
    E_{lin}^{(3)}(w(t),r(t))\le e^{4CM^2T} E_{lin}^{(3)}(w_0,r_0). \label{linearizedineqn}
\end{equation} 

From the error estimates~\eqref{gk-epsilon} in Theorem~\ref{t:ww-approx} and interpolation, we have 
\begin{equation*}
    \|w(s,s)\|_{H^1}\lesssim \epsilon^{\frac 7 2},\quad \|w(s,s)\|_{L^\infty}\lesssim \epsilon^{\frac 7 2},\quad \|q(s,s)\|_{\dot H^{\frac 1 2}}\lesssim\epsilon^{\frac 7 2},
\end{equation*}
while for $R(s,s)=\dfrac{Q_\alpha(s,s)}{1+\W(s,s)}$, the construction in the previous section guarantees
\begin{equation*}
    \|R(s,s)\|_{L^\infty}\lesssim \epsilon^2,\quad \|R(s,s)\|_{\dot H^{\frac 1 2}}\lesssim \epsilon^2.
\end{equation*}
Combining these, we see that
\begin{equation*}
    \|r(s,s)\|_{\dot H^{\frac 1 2}}\lesssim \|q(s,s)\|_{\dot H^{\frac 1 2}}+\|R(s,s)\|_{L^\infty}\|w(s,s)\|_{\dot H^{\frac1 2}}+\|R(s,s)\|_{\dot H^{\frac 1 2}}\|w(s,s)\|_{L^\infty}\lesssim\epsilon^{\frac 7 2}.
\end{equation*}
The estimate~\eqref{linearizedineqn} then gives
\begin{equation}
    \|\left(w(t,s),r(t,s)\right)\|_{\mathcal H}\lesssim e^{4CM^2T}\epsilon^{\frac 7 2}.
\end{equation}

Integrating in $s$, we have \begin{equation}
    \|W(t,s)-W(t,t)\|_{L^2}\lesssim Te^{4CM^2T}\epsilon^{\frac 3 2},\label{linbound:W}
\end{equation}
while for 
\begin{equation*}
    R_s=\frac{Q_s-RW_{\alpha s}}{1+W_\alpha}=\frac{r_\alpha+R_\alpha w}{1+W_\alpha}
\end{equation*}
we estimate, exploiting the fact that $r$, $w$, $R$, and $W$ are holomorphic and hence there are no high-high to low frequency interactions in the Littlewood-Paley trichotomy when two such factors are multiplied,
\begin{equation*}
   \left \|\frac{R_\alpha}{1+W_\alpha}w\right\|_{\dot H^{-\frac 1 2}}\lesssim \left\|\frac{R_\alpha}{1+W_\alpha}\right\|_{L^2}\|w\|_{L^2}\lesssim 2M\epsilon^{6}e^{4CM^2T}
\end{equation*} 
and
\begin{equation*}
    \left\|\frac{r_\alpha}{1+W_\alpha}\right\|_{\dot H^{-\frac 1 2}}\lesssim \|r_\alpha\|_{\dot H^{-\frac 1 2}}\left\|\frac{1}{1+W_\alpha}\right\|_{L^\infty}\lesssim e^{4CM^2T}\epsilon^{7/2}.
\end{equation*}
Integrating in $s$, we have
\begin{equation}
    \|R(t,s)-R(t,t)\|_{\dot H^{-\frac 1 2}}\lesssim Te^{4CM^2T}\epsilon^{\frac 3 2}.\label{linbound:R}
\end{equation}
Similarly, we have \begin{equation}
    \|Q(t,s)-Q(t,t)\|_{\dot H^{\frac 1 2}}\lesssim\|(1+\W(t,s))R(t,s)-(1+\W(t,t))R(t,t)\|_{\dot H^{-\frac 1 2}}\lesssim Te^{4CM^2T}\epsilon^{\frac 3 2}.\label{linbound:Q}
\end{equation}

To study the behavior of the family $(W(t,s),Q(t,s))$ in the $t$-direction, we now inductively apply Proposition~\ref{p:improvedenergyestimate} to get a bound for $E^{n,(3)}$ (and thus for $\|(\W, R)\|_{\doth{n}}$).

We claim that for $0\le n\le m-1$, we have the following bound: \begin{equation}
\|(\W(t,s),R(t,s))\|_{\doth{n}}\lesssim (4M^2T)^{n+1}\epsilon^{n+\frac 3 2}.\label{hfbound}
\end{equation}

By the norm equivalence, we can replace $E^{n,(3)}$ in our estimates with $\|(\W,R)\|_{\doth{n}}^2$.
From the estimate~\eqref{strongcubicest} and our bootstrap assumptions, we have 
\begin{align*}
    \frac{d}{dt}\|(\W(t,s),R(t,s))\|_{\doth{n}}^2\lesssim &4M^2 \epsilon^{\frac 5 2}\|(\W(t,s),R(t,s))\|_{\doth{n}}^2 \\
    &+4M^2\epsilon^3\|(\W(t,s),R(t,s))\|_{\dot{\mathcal H}_{n}}\|(\W(t,s),R(t,s))\|_{\doth{n-1}}.
\end{align*}
Dividing through by $\|(\W(t,s),R(t,s))\|_{\dot{\mathcal H}_{n}}$, we get 
\begin{equation}
    \frac{d}{dt}\|(\W(t,s),R(t,s))\|_{\doth{n}}\lesssim 4M^2\epsilon^{\frac 5 2}\|(\W(t,s),R(t,s))\|_{\doth{n}}+4M^2\epsilon^3\|(\W(t,s),R(t,s))\|_{\doth{n-1}}.\label{inductivestep}
\end{equation}

When $n=-1$, we know by the conservation of energy in the original $(W,Q)$ variables that $E^{-1,(3)}=\mathcal E(W(\cdot,s),Q(\cdot,s))=\mathcal E(W(s,s),W(s,s))\lesssim \epsilon$.
This gives us the base case for an induction.

For $0\le n\le m-1$, the frequency concentration of the solution $(W(s,s),Q(s,s))$ ensured by~\eqref{e:firststepfreqconcentration} guarantees that \begin{equation}
    \|(\W(s,s), R(s,s))\|_{\doth{n}}\lesssim \epsilon^{n+\frac 3 2},\label{e:improvedEE}
\end{equation}
so applying Gronwall's inequality on~\eqref{inductivestep}, we have 
\begin{equation}
    \|(\W(t,s),R(t,s))\|_{\doth{n}}\lesssim e^{4M^2CT\epsilon^{1/2}}\epsilon^{n+\frac 3 2}+4M^2 T\epsilon(4M^2T)^{n}\epsilon^{n+\frac 1 2}.
\end{equation}
As long as $\epsilon$ is much smaller than $M$ and $T$, the coefficient in the exponential is small, so the bound~\eqref{hfbound} is proven.

This means that for any $0\le n\le m$, we have 
\begin{align}
    \|W(t,s)-W(t,t)\|_{\dot H^n}\lesssim& (4M^2T)^{n}\epsilon^{n+\frac 1 2},\label{energybound:W}
\end{align} 
and for any $0\le n\le m-1$,
\begin{align}
    \|R(t,s)-R(t,t)\|_{\dot H^{n+\frac 1 2}}\lesssim&(4M^2T)^{n+1}\epsilon^{n+\frac 3 2}.\label{energybound:R}
\end{align}

Interpolating the high-frequency bounds~\eqref{energybound:W} and~\eqref{energybound:R} with the low-frequency bounds \eqref{linbound:W} and~\eqref{linbound:R}, we are able to close our bootstrap estimates:

For $\|W\|_{L^\infty}$ there is no problem, as both the low-frequency bounds and the $n=1$ high-frequency bound are of size $\epsilon^{\frac 3 2}$, so as long as $\epsilon<e^{-KT}$ for some large constant $K$, the $(4M^2T)^{\frac 1 2}$ and $\sqrt T e^{2CM^2T}$ coefficients can be defeated.

For $\||D|^{\frac 1 2}\W\|$, we interpolate between the low-frequency bound and the $n=m$ high-frequency bound for $W$.
Since the power of $\epsilon$ we will get is
\begin{equation*}
    \frac{m-2}{m}\cdot \frac 3 2+\frac{2}{m}\cdot \frac{2m+1}{2},
\end{equation*} 
as long as $m\ge 3$ and $\epsilon<e^{-KT}$, the coefficients can be absorbed by spending a power of $\epsilon$.

For the $\|R\|_{L^\infty}$ term, we interpolate between the low-frequency bound and the $n=m-1$ high-frequency bound.
Since the power of $\epsilon$ we will get is
\begin{equation*}
    \frac{m-1}{m}\cdot \frac 3 2+\frac{1}{m}\cdot \frac{2m+1}{2},
\end{equation*}
as long as $m\ge 3$ and $\epsilon<e^{-KT}$, the coefficients can be absorbed by spending a power of $\epsilon$.

For the $\|R_\alpha\|_{L^\infty}$ term, we interpolate between the low-frequency bound and the $n=m-1$ high-frequency bound.
Since the power of $\epsilon$ we will get is
\begin{equation*}
    \frac{m-2}{m}\cdot \frac 3 2+\frac{2}{m}\cdot \frac{2m+1}{2},
\end{equation*}
as long as $m\ge 3$ and $\epsilon<e^{-KT}$, the coefficients can be absorbed by spending a power of $\epsilon$.

The remaining components of $\uA$ and $\uB$ can be estimated from these four by interpolation, so as long as $m\ge 3$ and  $\epsilon<e^{-KT}$ for some large constant $K$, we can close our bootstrap and get~\eqref{bootstrapconclusion} on all of $[0,T/\epsilon^2]\times [0,T/\epsilon^2]$.
This restriction on $\epsilon$ justifies Remark~\ref{r:epsilon0} and gives the lifespan bound $T_\epsilon\approx \epsilon^{-2}|\ln \epsilon$.

In particular, the control norms $\uA$ and $\uB$ stay bounded, so the solution exists on the entire box $[0,T/\epsilon^2]\times [0,T/\epsilon^2]$.
The estimates~\eqref{linbound:W} and~\eqref{linbound:Q} give the desired closeness between the exact solution and the approximate solution, and the energy estimates~\eqref{hfbound} give the desired frequency concentration, proving~\eqref{e:secondstepestimates}.

This concludes the proof of Theorem~\ref{t:exactsolutionclose}.

\section{Stability estimates} \label{s:Stability}

In this section, we prove Theorem~\ref{t:perturbation}, using the frequency concentration of the initial data to set up a bootstrap argument similar to that in the previous section.
We conclude by showing that Theorem~\ref{t:perturbation} implies the alternate form of our main approximation result, namely Theorem~\ref{t:maintheoremWPform}.

\subsection{Proof of Theorem~\ref{t:perturbation}}
We let $(W^h_0, Q^h_0)_{h\in[0,1]}$ to be the linear interpolation of the frequency-concentrated initial data $(W^1_0, Q^1_0)=(W_0,Q_0)$ of Theorem~\ref{t:perturbation} and the initial data $(W_0^0, Q_0^0)$ built from $\tilde Y^\epsilon$ in Theorem~\ref{t:maintheorem}. 

The frequency concentration bounds~\eqref{e:perturbationfrequencyconcentration} guarantee that \begin{equation}
    \|\left(W^1_0,Q_0^1\right)\|_{\mathcal H}\lesssim \epsilon^{\frac 1 2},\quad \|\left(\partial^j \W_0^1,\partial^j R_0^1\right)\|_{\mathcal H}\lesssim \epsilon^{j+\frac 1 2},\quad j=0,m-1,
\end{equation} 
and the solution $(W^0,Q^0)$ constructed by~\eqref{t:maintheorem} has initial data equal to that of the approximate solution in Theorem~\ref{t:ww-approx}, which being constructed from $\tilde Y^\epsilon$, its derivatives, and $\nP[\tilde Y^\epsilon\overline{\tilde Y^\epsilon}_\alpha]$, satisfies
\begin{equation}
    \|\left(W^0_0,Q_0^0\right)\|_{\mathcal H}\lesssim \epsilon^{\frac 1 2},\quad \|\left(\partial^j \W_0^0,\partial^j R_0^0\right)\|_{\mathcal H}\lesssim \epsilon^{j+\frac 1 2},\quad j=0,m-1,
\end{equation} 
by~\eqref{e:mainyestimates}.
This means that \begin{equation}
    \|\left(W^h_0,Q_0^h\right)\|_{\mathcal H}\lesssim \epsilon^{\frac 1 2},\quad \|\left(\partial^j \W_0^h,\partial^j R_0^h\right)\|_{\mathcal H}\lesssim \epsilon^{j+\frac 1 2},\quad j=0,m-1,\label{e:whqhfrequencycontenctration}
\end{equation}
for all $0\le h\le 1$.

Since $W^0_0=2\tilde Y^\epsilon_0 + O_{L^2}\left(\epsilon^{\frac 3 2}\right)$ and $Q_0^0=\frac{2g}{c}\tilde Y^\epsilon_0+O_{\dot H^\frac 1 2}\left(\epsilon^{\frac 3 2}\right)$, combining the closeness condition~\eqref{e:perturbinitial} with the closeness of $\tilde Y^\epsilon$ to $Y^\epsilon$ from~\eqref{e:L2truncationerror}, we have
\begin{equation}
    \|\left(W_0^1-W_0^0,Q_0^1-Q_0^0\right)\|_{\mathcal H}\lesssim\epsilon^{1+\delta}.\label{e:perturbdifferenceinitial}
\end{equation}

\par The solutions of $\eqref{vorticityEqn}$ with initial data $(W^h_0, Q^h_0)$ are in $\mathcal H^m$, they exist locally in time near $t=0$, and they depend smoothly on $h$. 
We claim that they exist uniformly up to time $T_\epsilon$, and that their associated control
norms satisfy the bound
\begin{equation*}
    \uA(h,t) +\uB(h,t) \lesssim \epsilon.
\end{equation*}
By a continuity argument, it suffices to show this under the additional bootstrap assumption that 
\begin{equation*}
    \uA(h,t) +\uB(h,t) \lesssim M\epsilon
\end{equation*}
is true in $[0, t_0]$, with $t_0 \leq T_\epsilon$.

Our bootstrap argument will proceed as in the previous section.
From the assumptions on $(W^1,Q^1)$ in Theorem~\ref{t:perturbation}, we know that for an open set of $(h,t)$ including the whole interval $[0,1]\times \{0\}$, the following estimates hold: 
\begin{equation}
    \|W^h(t)\|_{L^\infty}\le M\epsilon,\quad \||D|^{\frac 1 2}\W^h(t)\|_{L^\infty}\le M\epsilon^{\frac 5 2},\quad \|R^h(t)\|_{L^\infty}\le M\epsilon^2,\quad \|R^h_\alpha(t)\|_{L^\infty}\le M\epsilon^{3}.
\end{equation}
By interpolation, these assumptions imply all the pointwise control norms are bounded, and in particular that $\uA\le M\epsilon$, $\uB \le M\epsilon^{\frac 3 2}$.

On the one hand, the frequency concentration estimates~\eqref{e:whqhfrequencycontenctration} guarantee 
\[
\|(\W^h,R^h)(0)\|_{\doth{n}}\lesssim \epsilon^{n+\frac 3 2},
\] 
for $0\le n\le m-1$, hence the improved energy estimates of Proposition~\ref{p:improvedenergyestimate} can be applied as long as the bootstrap assumptions hold.
This gives
\begin{equation}
\|(\W^h,R^h)\|_{\doth{n}}\lesssim (4M^2T)^{n+1} \epsilon^{n+\frac 3 2} \label{e:perturbdatahfbound}
\end{equation} for $0\le n\le m-1$.
Combining these with the corresponding bounds on $(W^0,R^0)$ from Theorem~\ref{t:maintheorem}, we have the high-frequency bounds
\begin{equation}
\|(\W^h-W^0,R^h-R^0)\|_{\doth{n}}\lesssim (4M^2T)^{n+1} \epsilon^{n+\frac 3 2},\quad 0\le n\le m-1 \label{e:differencehfbound}
\end{equation}

On the other hand, as in the argument in section \ref{s: bootstrap}, for the good linearized variable
$(w,r)$ given by
\begin{equation*}
    (w,q) = \partial_h(W^h, Q^h), \quad r=q-R^h w,
\end{equation*}
we have the linearized equation \eqref{linearizedeqn} around the solution $(W^h,Q^h)$ with initial data
\[
(w_0,r_0) = (W^1_0-W^0_0,Q^1_0-Q^0_0- R^h_0(W^1_0-W^0_0))
\]
Using~\eqref{e:perturbdifferenceinitial}, the frequency-concentration bounds~\eqref{e:whqhfrequencycontenctration}, and interpolation, we have 
\begin{equation*}
    \|W^1_0-W^0_0\|_{\dot H^\frac 1 2}\lesssim \epsilon^{\frac 3 2+\delta},\quad \|W^1_0-W^0_0\|_{L^\infty}\lesssim\epsilon^{\frac 3 2 + \delta},\quad \|Q^1_0-Q_0^0\|_{\dot H^{\frac 1 2}}\lesssim\epsilon^{\frac 3 2+\delta}
\end{equation*} 
while for $R^h=\dfrac{Q^h_\alpha}{1+\W^h}$ we have
\begin{equation*}
    \|\W^h_0\|_{L^\infty}\lesssim \epsilon,\quad \|R^h_0\|_{\dot H^{\frac 1 2}}\lesssim\epsilon^2,\quad \|R^h_0\|_{L^\infty}\lesssim\epsilon^2.
\end{equation*} 
Since \begin{equation*}
    \|R^h_0(W^1_0-W^0_0)\|_{\dot H^{\frac 1 2}}\lesssim \|R^h_0\|_{L^\infty}\|W^1_0-W_0^0\|_{\dot H^\frac 1 2}+\|R^h_0\|_{\dot H^\frac 1 2}\|W^1_0-W_0^0\|_{L^\infty},
\end{equation*}
we can put these bounds together to obtain the initial data bound
\begin{equation}
    \| (w(0),r(0))\|_{\mathcal{H}}\lesssim \epsilon^{1+\delta},
\end{equation}
Then we may apply the energy estimates for the linearized equation in Proposition~\ref{t:linearizedestimates} in order to obtain
\begin{equation}
    \| (w,r)\|_{\mathcal{H}}\lesssim e^{4CM^2T}\epsilon^{1+\delta},
\end{equation}
using the bootstrap assumptions on both $\uA$ and $\uB$, as well as the better bound on $\|R\|_{L^\infty}$ in order to get the improvement in Proposition~\ref{t:linearizedestimates} as before.

By integrating with respect to $h\in [0,1]$ and arguing as in the previous section, we get the low-frequency bounds
\begin{equation}
\begin{aligned}
    \| W^1 -W^0\|_{L^2}\lesssim e^{4CM^2T}\epsilon^{1+\delta}, &\qquad \| R^1 -R^0\|_{\dot H^{-\frac{1}{2}}}\lesssim e^{4CM^2T}\epsilon^{1+\delta},\\
    \|Q^1-Q^0\|_{\dot H^{\frac 1 2}}&\lesssim e^{4CM^2T}\epsilon^{1+\delta}.\label{e:stabilitylowfreqestimates}
    \end{aligned}
\end{equation}

We interpolate these bounds with the bounds~\eqref{e:differencehfbound}.

For $\|W\|_{L^\infty}$ we need only use the $n=0$ bound for the high-frequency estimates to get
\begin{equation*}
    \|W^h-W^0\|_{L^\infty}\lesssim (4M^2T)^{\frac 1 2}e^{2CM^2T}\epsilon^{\frac{5}{4}+\frac{\delta}{2}},
\end{equation*} which, as long as $\epsilon$ is small enough (as in the previous section), gives \begin{equation}
    \|W^h\|_{L^\infty}\le \epsilon.
\end{equation}

For $\||D|^{\frac 1 2}\W\|_{L^\infty}$, we interpolate the low-frequency bound with the $n=m-1$ high-frequency bound for $\W$.
Our power of $\epsilon$ is then \begin{equation*}
    \frac{m-2}{m}\left(1+\delta\right)+\frac{2}{m}\cdot \frac{2m+1}{2}.
\end{equation*}

For $\|R\|_{L^\infty}$, we interpolate the low-frequency bound with the $n=m-1$ high-frequency bound for $R$.
Our power of $\epsilon$ is then \begin{equation*}
    \frac{m-1}{m}\left(1+\delta\right)+\frac{1}{m}\cdot \frac{2m+1}{2}.
\end{equation*}

For $\|R_\alpha\|_{L^\infty}$, we interpolate the low-frequency bound with the $n=m-1$ high-frequency bound.
Our power of $\epsilon$ is then \begin{equation*}
    \frac{m-2}{m}\left(1+\delta\right)+\frac{2}{m}\cdot \frac{2m+1}{2}.
\end{equation*}

Since the corresponding bounds were shown for $(\W^0,R^0)$ in Section~\ref{s: bootstrap}, as long as $\frac{1}{2m-2}<\delta$ and as long as $\epsilon$ is small enough, we have \begin{equation}
    \|W^h\|_{L^\infty}\le \epsilon,\quad \||D|^{\frac 1 2}\W^h\|_{L^\infty}\le \epsilon,\quad \|R\|_{L^\infty}\le \epsilon^2,\quad \|R_\alpha\|_{L^\infty}\le \epsilon^{\frac 5 2}.
\end{equation}
This closes our bootstrap and guarantees that the solutions exist uniformly up to time $T_\epsilon$ with their control norms satisfying the bound \begin{equation*}
    \uA(h,t)+\uB(h,t)\lesssim\epsilon
\end{equation*} as desired.

Since the solution $(W^1,Q^1)$ exists on the whole time interval $[0,T_\epsilon]$, the difference estimates~\eqref{e:stabilitylowfreqestimates} and the higher-regularity estimates~\eqref{e:perturbdatahfbound} hold for the perturbed solution on $[0,T_\epsilon]$, concluding the proof of Theorem~\ref{t:perturbation}.

\subsection{Proof of Theorem~\ref{t:maintheoremWPform}}
Given $\epsilon$-well-prepared initial data $(W_0,Q_0)$, set $Y^{\epsilon}_0=\dfrac{W_0}{2}$.
Let $Y^\epsilon$ be the solution to the negative-frequency Benjamin-Ono equation~\eqref{negativeFrequencyBO} with $\lambda=c$ and initial data $Y^\epsilon_0$.
The bounds for the corresponding Benjamin-Ono initial data $U_0$ follow from the bounds on $W_0$ after untangling the transformations and rescalings.

Since $(W_0,Q_0)$ is $\epsilon$-well-prepared, $Y^\epsilon$ is in $H^m$ and the conditions~\eqref{e:perturbinitial} and~\eqref{e:perturbationfrequencyconcentration} required to apply Theorem~\ref{t:perturbation} apply with $\delta=\frac 1 2$.
Because we are working with $\delta=\frac 1 2$, the conclusion~\eqref{e:lfperturbationbound} gives the $\epsilon^{\frac 3 2}$ closeness in $\mathcal H$ for $(W,Q)$ and the conclusion~\eqref{e:hfperturbationbound} gives the desired higher regularity estimates for $(\W,R)$.
This concludes the proof of Theorem~\ref{t:maintheoremWPform}.

\bibliographystyle{plain}


\end{document}